\documentclass[]{amsart}

\usepackage{amsmath,dsfont}
\usepackage{amsfonts}
\usepackage{amssymb}
\usepackage{amsthm}
\usepackage{mathrsfs}
\usepackage{enumerate}
\usepackage{graphicx,color,xcolor,tikz}
\usepackage{bbold}
\usepackage[utf8]{inputenc}
\usepackage{hyperref}
\hypersetup{
    colorlinks=true,                         
    linkcolor=blue, 
    citecolor=red, 
  } 
\usepackage{enumitem,a4wide}

\newtheorem{Thm}{Theorem}
\newtheorem{Lem}[Thm]{Lemma}
\newtheorem{Prop}[Thm]{Proposition}
\newtheorem{Cor}[Thm]{Corollary}

\newcommand{\dv}{\partial}
\newcommand{\p}{\partial}
\newcommand{\brho}{{\bar\rho}}
\newcommand{\bu}{{\bar u}}
\newcommand{\tu}{{\tilde u}}
\newcommand{\bS}{{\bar S}}
\newcommand{\balpha}{{\bar \alpha}}
\newcommand{\bSigma}{{\bar \Sigma}}
\newcommand{\tSigma}{{\tilde \Sigma}}

\newcommand{\subscript}[2]{$#1 _ #2$}

\title[Analysis of compressible bubbly flows. Part II.]{Analysis of compressible bubbly flows. \\Part II : Derivation of a macroscopic model.}
\date{\today}
\author{Matthieu Hillairet}
\address{Institut Montpelli\'erain Alexander Grothendieck, Univ Montpellier, CNRS, Montpellier, France}
\email{matthieu.hillairet@umontpellier.fr}

\author{Hélène Mathis}
\address{Nantes Université, CNRS, Laboratoire de Math\'ematiques Jean Leray (LMJL - UMR 6629), 
  F-44000 Nantes, France}
\email{helene.mathis@univ-nantes.fr}

\author{Nicolas Seguin}
\address{Universit\'e de Rennes, Irmar, UMR CNRS 6625, 35042 Rennes Cedex, France}
\email{nicolas.seguin@univ-rennes1.fr}

\begin{document}
\begin{abstract}
  This paper {is the second of the series of two papers, which}
  focuses on the derivation of an averaged 1D model  for
  compressible bubbly flows. For this, we start from a microscopic
  description of the interactions between a large but finite number of
  small bubbles with a surrounding compressible fluid.
  {This microscopic model has been derived and analysed in the
  first paper. In the present one,}
provided physical
parameters scale according to the number of bubbles, we prove that
solutions to the microscopic model exist on a timespan  independent
of the number of bubbles. Considering then that we have a large number of bubbles, we propose a construction of the macroscopic
variables and derive the averaged system satisfied by these
quantities. Our method is based on a compactness approach in a strong-solution setting.   {In the last section, we propose the derivation of the Williams-Boltzmann equation corresponding to our setting. }
\end{abstract}

\maketitle


\noindent
\textbf{2020 MCS.}  {76T05, 76T10, 35Q30}

\setcounter{tocdepth}{3}
\tableofcontents

\section{Introduction}
\label{sec_introduction}

The present work represents a straight continuation of a series of articles which
proposes to justify the construction of multiphase flow
models.
The structure of multiphase flow models can be derived formally by
applying standard conservation principles
\cite{DrewPassman99,EmbidBaer92,Gavrilyuk14,Ishii}.
However this procedure leaves aside  key-terms that have to be related
to mechanical/thermodynamical unknowns \emph{via} state laws.
To this end, a sharp description of the interactions between phases is
required.
Classical methods are based on {averaging} operators whose range of
validity is still to be investigated.
Furthermore, the action of these {averaging} operators on nonlinear quantities
requires further modelling assumptions.
From the analytical standpoint, the computations we provide herein
follow previous analysis of the first author {notably} in collaboration with
D. Bresch \cite{BrHibook,BrHi1,BrHi2,H19} {complementing previous approaches
in \cite{AZ98,Hil07,PS12}}.
In these  references, one-velocity Baer-Nunziato-like models are
derived for multiphase fluids. These computations are based on the
remark that, if the interfaces act as a "perfect" transducer (no mass
transfer, perfect transfer of mechanical stress),
combining the different phases equations yields a global one-fluid
equation.
Deriving multiphase flow models then reduces  to a thorough analysis
of highly-oscillatory solutions to the one-fluid equation.
A particular analytical  framework of mixed-regularity (smooth
velocity with discontinuous densities \cite{Desjardins97,Hoff95,Serre91}) is 
identified in \cite{BrHuang} to make this approach fully rigorous.
However, this approach is restricted to an ideal case (see
\cite{BrBuLa} for further investigations in this context).
The aim of this paper is to tackle the derivation of averaged models
in presence of jumps at interfaces.  Starting from an original microscopic model (that is {derived in the first paper \cite{HMS1}})  in which the two phases are fully separated, we derive a 1D 
averaged compressible bubbly-flow model by performing
{space averaging} operators.

\medskip

The averaged model reads as follows. It is set on the container
$\Omega = (-1,1)$ filled with a  gas/fluid mixture.
The averaged variables are the void fractions $\balpha_{f,g} \in
[0,1]$, the mean densities $\brho_{f,g} \in [0,\infty)$,
a bubble phase covolume\footnote{{The denomination covolume may be
    misleading here. In classical thermodynamic, the term covolume refers
    to the specific volume. Here the quantity $\bar f_g$ is linked to
    the volume of the gaseous phase. 
    In 3D configurations,    
    it would be related to the interfacial area.}} $\bar{f}_g \in [0,\infty)$ and the
mixture velocity $\bu \in \mathbb R$. It reads:

\begin{equation}
  \label{eq_macromodel}
    \begin{cases}
    \dv_t ( \balpha_g\bar f_g) + \dv_x (\balpha_g \bar f_g \bar u) = 0, \\
        \dv_t (\balpha_f \brho_f) + \dv_x (\balpha_f \brho_f \bu) = 0 , \\
    \dv_t (\balpha_g \brho_g) + \dv_x (\balpha_g \brho_g \bu) = 0 , \\[4pt]
       \displaystyle \dv_t \balpha_f + \bu \dv_x \balpha_f = \overline{RT}
    , \\
    \dv_t (\brho \bu) + \dv_x (\brho \bu^2) = \dv_x \bSigma ,
  \end{cases}
  \qquad 
  \text{ on $(0,T) \times \Omega$},
\end{equation}
with the compatibility conditions:
\begin{equation} \label{eq_comp}
    \balpha_f + \balpha_g = 1 , \quad
     \brho = \balpha_f \brho_f +    \balpha_g \brho_g ,
\end{equation}
and where the mixture stress tensor writes
\begin{equation}
  \label{eq_bSigma}
  \bSigma =     \frac{\mu_g\mu_f}{\balpha_f \mu_g + \balpha_g \mu_f}
  \bigg[ \dv_x \bu - \bigg(\frac{\balpha_f}{\mu_f}
  \mathrm{p}_f(\brho_f) +  \frac{\balpha_g}{\mu_g}
  \mathrm{p}_g(\brho_g) \bigg) - \bar \gamma_s \frac{\balpha_g }{\mu_g} \bar f_g\bigg],
\end{equation}
while  the void fraction relaxation term reads:
\begin{equation} \label{eq_bT}
\overline{RT} = \frac{\balpha_g\balpha_f}{\balpha_f \mu_g + \balpha_g \mu_f} \bigg[
    (\mu_g-\mu_f) \dv_x \bu + (\mathrm{p}_f(\brho_f) -
    \mathrm{p}_g(\brho_g) ) - \bar
    \gamma_s \bar f_g \bigg].
\end{equation}
In these latter identities appear the constants $\mu_f,\mu_g >0$
(resp. the functions ${\rm p}_f,{\rm p}_g$) representing the fluid and
bubble viscosities (resp. the fluid and gas pressure laws). The
constant $\bar{\gamma}_s >0$ represents the surface tension.

\medskip

The system \eqref{eq_macromodel}-\eqref{eq_comp} complemented with the
state laws \eqref{eq_bSigma}-\eqref{eq_bT} is obtained starting from
the following
\textit{microscopic} model, where the two phases are disjoint and
their interactions only appear through
the interfaces. Again, the two-phase flow is posed in the
one-dimensional domain $\Omega=(-1,1)$, filled by a liquid (the fluid,
or the continuous phase, indexed by $f$) and bubbles (the gas, or the
dispersed phase, indexed by $g$). The $N$ bubbles are described by their
centers $c_k$ and their radii $R_k$, so that the $k$-th bubble is
\[
  B_k = (x_k^-,x_k^+), \quad x_k^\pm = c_k\pm R_k, \quad \forall \, k =
  1,\ldots,N.
\]
The fluid domain is 
\[
  \mathcal F = \Omega \setminus \bigcup_{k=1}^N \overline{B_k}.
\]
For later use, we also introduce the fluid intervals
\begin{equation} \label{eq_Fk}
\mathcal F_k=(x_k^+,x_{k+1}^-) \quad  \text{for $k=0,\dots,N$}
\end{equation}
setting $x_0^+=-1$ and $x_{N+1}^-=1$.

\medskip

The fluid is supposed to be compressible and viscous, so that it is
governed by the 1D compressible Navier-Stokes system, posed in
$\mathcal F$:
\begin{align}
  \label{eq_fluid_mass}
  &\partial_t \rho_f + \partial_x ( \rho_f u_f )  = 0 , \\
  \label{eq_fluid_NS_1D}
  &\partial_t (\rho_f u_f) +\partial_x (\rho_f u_f^2) =  \partial_x
  \Sigma_f , \\
  \label{eq_fluid_tensor_sigmaf}
  &\Sigma_f  = \mu_f \partial_x u_f -  {\rm p}_f(\rho_f) ,
\end{align}
where $\rho_f$ is the density, $u_f$ the velocity and $\Sigma_f$ the
stress tensor of the fluid. Moreover, $\mu_f>0$ is the shear viscosity
and ${\rm p}_f$ is an isentropic pressure law
for the fluid:
\[
  \mathrm{p}_f(\rho_f) = \kappa_f \rho_f^{\gamma_f} ,
\]
where $\kappa_f>0$ and $\gamma_f>1$ {stands for the adiabatic exponent}.
We assume that the fluid is present at the boundary of the domain $\Omega$,
where no-slip boundary conditions are imposed:
\begin{equation}
  \label{eq_BC_uf}
  u_f(t,\pm1) = 0 .
\end{equation}
Equations for bubble kinematics and dynamics are proposed in \cite{HMS1}.  Therein, the derivation is based on the assumption that the bubbles are made of a compressible viscous fluid with an infinite shear viscosity (compared to the volumic viscosity) and that their spherical
shapes are preserved  (in three
dimensions).  This yields first that the continuity of
the velocity at the interfaces reads:
\begin{equation}
  \label{eq_continuity_velocity}
  u_f(t,x_k^\pm(t)) = \dot{c}_k(t) \pm \dot{R}_k(t) \quad \text{ for }
  k = 1,\ldots,N .  
\end{equation}
In addition, imposing that the jump of the stress tensor at the
interfaces is due to the surface tension, one obtains the following
system for the dynamics of a bubble:
\begin{align}
  \label{eq_droplet_newton_1D_C}
  &m_k \ddot c_k(t) = \Sigma_f(t,x_k^+) - \Sigma_f(t,x_k^-) , \\
  \label{eq_droplet_newton_1D_R}
  &\displaystyle \frac{m_k}{3} \ddot R_k(t) = \Sigma_f(t,x_k^-) +
  \Sigma_f(t,x_k^+)  - 2 \Sigma_k(t) , \\
    \label{eq_droplet_tensor_sigmak}
   &\displaystyle{ \Sigma_k = \mu_g \frac{\dot R_k}{R_k} - {\rm
     p}_g(\rho_k) - \dfrac{F_s}{2},}
\end{align}
where $\mu_g>0$ is the volumic viscosity of the gas, and $m_k$ and
$\rho_k$ are the mass and the density of the bubble, linked by
$m_k=2R_k\rho_k$. As a consequence of mass conservation in bubbles,
the  masses $m_k$ do not depend on time. The term $F_s$ denotes
the force due to the surface tension and writes $F_s=\gamma_s/R_k$,
$\gamma_s$ being the surface tension. In order to simplify the
analysis, we assume an isothermal equation of state in the bubbles, so
that
\begin{equation}
  \label{eq_droplet_pk}
  \pi_k := \mathrm{p}_g(\rho_k)+ \dfrac{F_s}{2} =
  \frac{(a_g)^2m_k + \gamma_s/2}{R_k} = \frac{\kappa_k}{R_k} , 
\end{equation}
where $a_g>0$ is the sound speed of the gas.
{The last form of $\pi_k$ will be used mainly for the analysis of
  the model, while the first form will be useful to interpret the various terms
  appearing in equations, notably those due to
  surface tension.  In particular,  computing surface tension effects in the microscopic system involves the quantity $1/(2R_k)$ that corresponds to the covolume of bubble $B_k$
  in our $1D$ setting.}
We point out that the system
\eqref{eq_fluid_mass}-\eqref{eq_droplet_tensor_sigmak} is not
integrable and,
specifically, does not yield any particular value for the fluid
velocity-field $u_f$.
We are then not in the Rayleigh-Plesset regime where the bubble
equations
\eqref{eq_droplet_newton_1D_C}-\eqref{eq_droplet_newton_1D_R} reduce
to ordinary differential equations in terms of $(c_k,R_k)$ and an
asymptotic pressure \cite{Wang-Smereka-03}.
We refer the reader to the companion paper \cite{HMS1} for more details on the derivation of \eqref{eq_fluid_mass}-\eqref{eq_droplet_tensor_sigmak} and the analysis of the associated Cauchy problem.  Yet,  we shall explain in further details the construction of solutions in the next section.

\medskip

The main result of this paper is to show that,  starting from solutions to 
\eqref{eq_fluid_mass}--\eqref{eq_droplet_tensor_sigmak} we obtain \eqref{eq_macromodel}--\eqref{eq_bT} by letting the number $N$ of bubbles go to infinity in case:
\begin{equation} \label{eq_scaling}
  m_k \sim N^{-1}, \quad
  R_k \sim N^{-1}, \quad
  |\mathcal F_k| \sim N^{-1}  \quad 
  \gamma_s \sim N^{-1} ,
\end{equation}
with the other parameters being fixed.
One key-difficulty in the proof is that 
the target system \eqref{eq_macromodel}--\eqref{eq_bT}  is highly nonlinear.  Specifically, products between volume fractions and other (fluid or gas) unknowns are ubiquitous.  To obtain such nonlinear terms,  it appears that strong convergences  of densities or gas covolume in sufficiently smooth spaces are necessary.  Hence, with this approach,  we face two key-difficulties: 
\begin{itemize}
\item to prove that the scaling regime \eqref{eq_scaling} holds on a timespan independent of the number $N$ of bubbles, 
\item to define the macroscopic unknowns and especially, the fluid and gas densities $\bar{\rho}_f,\bar{\rho}_g$ and the gas covolume $\bar{f}_g.$
\end{itemize}

The first item in this list is the content of the topic of the next section.  Therein, we 
consider initial data  that are constructed as follows.  Firstly, we fix fluid initial data $(\rho_f^0,u_f^0) \in H^1(\Omega) \times H^1_0(\Omega)$ that are thus defined globally on $\Omega.$ We assume further that they are far from vacuum.  Secondly, we fix initial distributions of centers/radii $(c_k^0,R_k^0)_{k=1,\ldots,N}$ such that \eqref{eq_scaling} holds. We complement then the microscopic system
\eqref{eq_fluid_mass}--\eqref{eq_droplet_tensor_sigmak}  with initial
conditions so that the initial bubble velocities match the velocities prescribed by the fluid on the boundaries. This reads:
\begin{align}
  \label{eq_initbubble}
  & c_k(0) = c_k^0 \quad R_k(0) = R_k^0, &&  \text{ for $k=1,\ldots,N$},\\
  \label{eq_initfluid}
  & u(0,\cdot) = u_f^0 \quad \rho(0,\cdot) = \rho_f^0, && \text{ on $\mathcal F^0$},
\end{align}
and
\begin{align}
  \label{eq_initbubble1}
  & \dot{c}^0_k = \dfrac{u_f^0(c_k^0 + R_k^0) + u_f^0(c_k^0 - R_k^0)}{2}, && \text{ for $k=1,\ldots,N,$}
  \\  \label{eq_initbubble2}
  & \dot{R}^0_k = \dfrac{u_f^0(c_k^0 + R_k^0) - u_f^0(c_k^0 - R_k^0)}{2}, && \text{ for $k=1,\ldots,N.$}
\end{align}
The main result of Section \ref{sec_estimates} is then that there exists a classical solution to \eqref{eq_fluid_mass}--\eqref{eq_droplet_tensor_sigmak} on a timespan that depends only on fluid initial data and the parameters quantifying initially assumption \eqref{eq_scaling}. 
To obtain this result, we combine classical energy and regularity estimates for Navier Stokes equations. We remind that, in this strategy, one classically uses extra regularity thanks to the form of the stress tensor $\Sigma_f.$ However, such regularity estimates should depend on the geometry (and then on $N$). To overcome this difficulty, we propose to consider suitable extensions of $\Sigma_f$ (resp. $\Sigma_g$) on the complementary gas (resp.  fluid) domain.  In this way, the extension is defined on a fixed domain and the regularity gain is independent of the geometry. We point out here that
contrary to the classical approach in the topic of homogenization of
multidimensional compressible Navier Stokes equations in perforated
domains \cite{YS18,FNT10},  our construction takes advantage of the
information on the moments on $\partial B_k$
of the fluid stress tensor  that are provided by the bubble
equations. 

The second key-difficulty of our approach is tackled in {Section \ref{sec_homog-probl}}.  Once solutions to \eqref{eq_fluid_mass}-\eqref{eq_droplet_tensor_sigmak} are constructed on a time-interval that does not depend on $N,$ we consider the behavior of these solutions for large $N$. In particular, we look for definitions of the unknowns that are involved in the macroscopic system  \eqref{eq_macromodel}--\eqref{eq_bT}.  Volumic fractions as well as global velocity-fields are obtained classically by considering indicator functions or suitably extended vector-fields (see Proposition \ref{prop_transport_chi} and Proposition \ref{prop_CV_uf}).  However,  the issue is more involved when going to density and covolume unknowns.  Indeed,  at the discrete level,  fluid density and bubble density, for instance, are defined {\em a priori} on dispersed subdomains only. This cannot yield convergence with sufficient regularity. To get better convergence results,  we decide to construct suitable extensions. 
For this we proceed in two steps. Firstly, we ensure that the initial conditions for   \eqref{eq_fluid_mass}-\eqref{eq_droplet_tensor_sigmak} enable to define smooth extended densities and covolume (see Proposition \ref{prop_cstr_id}). Then, we propagate this regularity with a well-chosen extended flow (see Proposition  \ref{prop_tilde_rho_fN} and Proposition \ref{prop_tilde_rhogN_fgN}). 
With this construction at-hand, the derivation of \eqref{eq_macromodel}--\eqref{eq_bT} is plain sailing.  

\medskip

In our construction, we start from initial data for the macroscopic system and define a sequence of initial conditions for the microscopic system that are compatible with the scaling \eqref{eq_scaling} and enable to construct extended densities. 
It turns out that this requires further assumption on initial data that we explain now. 
We recall that initial data for the macroscopic system consists in:
\begin{itemize}
\item initial  fluid and gas densities : $\brho_f^0,\brho_g^0$,
\item initial fluid and gas void fractions $\balpha_f^0,\balpha_g^{0}$ 
\item an initial velocity of the two-phase mixture  $\bu^0$,
\item an initial gas covolume $\bar{f}_g^{0}.$
\end{itemize}
It is worth noting that all these functions are defined for $x$ in
$\Omega$, since both phases are no longer separated at the macroscopic
scale.   We shall remain at the regularity level of classical solution and require
that all these initial conditions are $H^1(\Omega).$ 
For our construction, we require that initial densities and volumic fraction 
satisfy:
\begin{align} \label{eq_cond0}
&  \rho_{min} \leq \min ( \brho_f^0,  \brho_g^0)  \\
&  \alpha_{min} \leq \min ( \bar{\alpha}_f^0,\bar{\alpha}_g^0) && \bar{\alpha}_g^0 + \bar{\alpha}_f^0 = 1 
\label{eq_cond1}
\end{align}
for some strictly positive constants $\rho_{min},\alpha_{min}.$ The first condition means that we are away from void. The second one  expresses that there is a mixture of both phases everywhere in $\Omega.$ Note that the second conditions implie simultaneously that 
\begin{equation} \label{eq_cons0}
\max (\|\balpha_g^0\|_{L^{\infty}(\Omega)}, \|\balpha_f^0\|_{L^{\infty}(\Omega)}) \leq 1- \alpha_{min}.
\end{equation}
Concerning,  $\bar{f}_{g}^0,$ we will require that:
\begin{align} \label{eq_cond2}
& f_{min} \leq  \bar{f}_g^0,   \qquad  \balpha_g^0 \bar{f}_g^0  \in \mathbb P(\Omega), 
\end{align}
where $f_{min}$ is a strictly positive constant.  To explain these latter conditions, we point out that in the 1D case the covolume of bubbles is proportional to the inverse radius.  So $f_{min}$ is a bound from above on the initial radius of bubbles and, since we expect $\balpha_{g}^0 \bar{f}_g^0$
to be the limit of the indicator function of bubble domains multiplied by the inverse radius of bubbles, a straightforward computations yields that it is a positive function whose total mass is $1$, hence a probability density.  

\medskip

{Though we propose a converse interpretation to the classical one,  the multiphase  system we consider in this paper enters the family of sprays as  studied by Williams in \cite[Section 11]{Williams}.  With this standpoint,  a classical tool to analyze the behavior of the dispersed phase is the so-called "Williams-Boltzmann" equation which describes the time-evolution of the particle-distribution function of the dispersed phase. 
In the last section of this paper, we derive what would be the equivalent equation in our setting.  It is worth to mention that this is no supplementary equation but simply a rephrasing of the bubble-gas equation that we derived previously. In particular,  herein the bubble-gas velocities are correlated to their position and drag forces are at equilibrium.  We do neither have collision or creation of bubbles. Hence, the only term to be taken into account is the "evaporation" term which should be understood as compression/expansion term herein in our compressible setting. 
}

\medskip

In brief, the outline of the paper is as follows.
In the next section, we prove that solutions to the microscopic system
\eqref{eq_fluid_mass}--\eqref{eq_droplet_tensor_sigmak} with well-prepared initial
data do exist on a timespan 
independent of $N$, see {\bf
  Theorem \ref{thm_existence}}.
In {\bf Section \ref{sec_homog-probl}} and {\bf Section
  \ref{sec_deriv-macr-model}}, we tackle the asymptotics of these
solutions when $N \to \infty$.
In the last section, we discuss an alternative approach based on using particle-distribution functions for the bubbles.
In appendices, we provide some technical computations involved in the
construction of solutions to the microscopic model.\

\medskip

{\bf Acknowledgement.} 
The first author acknowledges support of the Institut Universitaire de France and project "SingFlows" ANR-grant number: ANR-18-CE40-0027.   This paper was finished while M.H. was benifiting a "subside à savant" from Université Libre de Bruxelles.  He would like to thank  the mathematics department at ULB for its hospitality.

\section{Local Cauchy theory for the microscopic system}
\label{sec_estimates}

{In this section,  we forget temporarily our homogenization goal.  We focus on the microscopic model \eqref{eq_fluid_mass}--\eqref{eq_droplet_tensor_sigmak} in the scaling \eqref{eq_scaling} and we address the existence of solutions with lifespan independent of the number $N$ of bubbles provided initial data are constructed as in \eqref{eq_initbubble}--\eqref{eq_initbubble2}.  In particular, we fix ($\rho_f^0,u_f^0) \in H^1(\Omega) \times H^1_0(\Omega)$ throughout the section.  We assume these global fluid data satisfy:
\begin{equation}
  \label{it_bound_rho0}
  2\underline{\rho}_{\infty} \leq \rho_f^0 \leq \overline{\rho}_{\infty}/2 \text{ on $\Omega$}
\end{equation}
for some pair $(\underline{\rho}_{\infty},\overline{\rho}_{\infty}) \in (0,\infty)^2.$}

\medskip

\medskip

{To make precise our main result, we start by giving a quantified version of assumption \eqref{eq_scaling} that we assume to hold initially.}  Firstly, we fix that bubbles characteristics enjoy the property:
\begin{enumerate}[label=(\subscript{IC}{{\arabic*}})]
  \setcounter{enumi}{-1}
\item \label{it_bound_mk0} $M_\infty\leq N m_k , N \kappa_k \leq
  (M_\infty)^{-1},$ $k=1,\ldots,N$,\\[-8pt]
\item\label{it_bound_R0}
  $2d_\infty\leq N R_k^0\leq (2d_\infty)^{-1}$, $k=1,\ldots,N$,  \\[-8pt]
\item \label{it_bound_F0}
  $2d_\infty\leq N|\mathcal{F}_k^0| \leq (2d_\infty)^{-1}$,
  $k=0,\ldots,N$,
\end{enumerate}
Here $M_{\infty},d_{\infty}$ are strictly positive constants independent of $N.$ 
We recall the convention \eqref{eq_Fk}
for the definition of $\mathcal F_k^{0}$ (adapted to notations for initial data).
Their union constitutes the initial fluid domain $\mathcal F^0$.
The physical parameters $(\mu_f,\mu_g)$ and pressure laws are fixed
independent of $N$. 
With these conventions, the main result of this section reads:
\begin{Thm}
  \label{thm_existence}
  Let initial conditions to
  \eqref{eq_fluid_mass}--\eqref{eq_droplet_tensor_sigmak} be
  constructed as in \eqref{eq_initbubble}--\eqref{eq_initbubble2}.
  Assume further that 
  parameters  $(m_k,\kappa_k)_{k=1,\ldots,N}$  and  initial bubble
  distributions $(c_k^0,R_k^0)_{k\in \{1,\dots, N\}}$ satisfy
  ~\ref{it_bound_mk0}--\ref{it_bound_F0}.
  Then, there exists $T_{\infty} >0$ depending only on 
  \begin{equation}
    \label{eq_listparametre}
    M_{\infty},d_{\infty},\underline{\rho}_{\infty},\bar{\rho}_{\infty},\|u_f^0\|_{H^1(\Omega)},\|\rho_f^0\|_{H^1(\Omega)},
  \end{equation}
  such that there exists a solution to
  \eqref{eq_fluid_mass}--\eqref{eq_droplet_tensor_sigmak}  on
  $(0,T_{\infty}).$
\end{Thm}

What remains of this section is devoted to the proof of this theorem.
From now on, we pick a family of physical parameters
and bubble centers/radii satisfying the assumptions of {\bf Theorem
  \ref{thm_existence}} and we construct initial data for  \eqref{eq_fluid_mass}--\eqref{eq_droplet_tensor_sigmak}.
  
\medskip
  
 In the companion paper \cite{HMS1}, we prove local-in-time existence
and uniqueness of classical solutions to the Cauchy problem associated
with \eqref{eq_fluid_mass}--\eqref{eq_droplet_tensor_sigmak}.  In this
moving-domain setting, classical solution means broadly that:
\begin{itemize}
\item the motion of the bubbles is $H^2(0,T)$ ({\em i.e.} $(c_k,R_k) \in H^2(0,T)$),
\item $u$ is $H^1_tL^2_x$ and $L^2_t H^2_x$ in the fluid domain,
\item $\rho$ is $H^1_{t,x}$ in the fluid domain.
\end{itemize}
Existence and uniqueness of solutions on a lifespan $(0,T_0)$ is obtained for initial data such that
\begin{itemize}
\item there is no overlap of the bubbles,
\item initial fluid data are $H^1$ in the fluid domain with strictly positive density,
\item initial fluid and bubble velocities match at interfaces (so that \eqref{eq_initbubble1}--\eqref{eq_initbubble2} hold true).
\end{itemize}
It is also worth noting that the time $T_0$ is uniform in data satisfying uniform bounds from below for the distance between bubbles, the minimal radius of
bubbles, the minimum density and also the size of initial fluid
velocity and density in $H^1$-spaces.  We refer to \cite{HMS1} for
more precise and quantitative statements.

\medskip

So, under the assumptions of {\bf Theorem \ref{thm_existence}}, the local-in-time existence result of \cite{HMS1} 
yields a solution on a time-interval $(0,T_0)$ that depends on the
list of parameters \eqref{eq_listparametre} but also on $N.$
To rule out this dependency, we construct $T_{\infty}$ such that as long as $t < T_{\infty}$ the
solution 
\[
  (\rho_f(t,\cdot),u_f(t,\cdot),(c_k(t),R_k(t), \dot c_k(t), \dot R_k(t))_{k\in
    \{1,\dots, N\}})
\] 
 yields an  initial condition that is compatible with the Cauchy theory of \cite{HMS1}
 with an associated existence time independent of $t.$
 We emphasize that any classical solution does not allow overlap of the bubbles and ensures identity \eqref{eq_initbubble1}--\eqref{eq_initbubble2} is satisfied at any time.   Controlling the existence time associated with the value of the solution at time $t$ -- considered as an initial data -- reduces to obtaining uniform $H^1$ 
 bound for the velocity
 field and for the density,  uniform bound from above and from below on the fluid density,  the radius of the bubbles and the length of fluid segments.
 
\medskip
 
Our approach relies on a suitable combination of energy and regularity
estimates for the coupled system
\eqref{eq_fluid_mass}--\eqref{eq_droplet_tensor_sigmak}. So,  we
recall in the next sections the classical estimates that are
associated with
\eqref{eq_fluid_mass}--\eqref{eq_droplet_tensor_sigmak}. We will pay
special attention to obtain estimates independent on $N.$ This
will be particularly challenging for regularity estimates. In
particular, we shall study the regularity of fluid velocity-fields
that can be gained through the integrability of the stress tensor by
working on extensions of fluid unknowns on bubble domains and
conversely. {A tricky part of the proof is that we can obtain these sharp bounds under the condition that we have already {\em a priori} bounds.  
So, we implement a continuation argument.  
This continuation argument is explained in the last part of the section.
However, the extensive proof is rather
long and technical. Hence, the last subsection reduces to a roadmap of the proof that is detailed further in
{\bf Appendix \ref{app_proof-prop-refpr}}.}
  
\subsection{Classical estimates}
\label{sec_classical-estimates}

We introduce the function $\mathrm{q}_f:[0,\infty) \to [0,\infty) $ defined by
\begin{equation}
  \label{eq_qTL}
  \mathrm{q}_f'(s)s-\mathrm{q}_f(s) = \mathrm{p}_f(s),
\end{equation}
which is conjugate of the fluid pressure. In other words, the function
$\mathrm{q}$ represents the volumic internal energy of the fluid.
Considering an isentropic pressure law, it yields
\begin{equation*}
  \mathrm{q}_f(s) =\dfrac{ a_f s^{\gamma_f}}{\gamma_f-1}.
\end{equation*}

We can now state the total energy equation. In the bracket of the statement below, the first
term is the total energy of the fluid, while the second and the third
terms respectively are the kinetic energy and the internal energy of
the bubbles.
\begin{Prop}
  \label{prop_estim_nrj}
  For any reference radius $R_\mathrm{ref}>0$, it holds
  \begin{equation}
    \label{eq_estim_nrj}
    \begin{aligned}
      \dfrac{\mathrm{d}}{\mathrm{d}t} &\left[ \int_{\mathcal F}
        \left(\rho_f \dfrac{|u_f|^2}{2}+\mathrm{q}_f(\rho_f) \right)
        \mathrm{d}x + \sum_{k=1}^N m_k \left(\dfrac{|\dot
            c_k|^2}{2} + \dfrac{|\dot R_k|^2}{6}  \right) \right.\\
      &\left.-{2} \sum_{k=1}^N \kappa_k \ln\left(
          \dfrac{R_k}{R_\mathrm{ref}}\right) \right] + \int_{\mathcal
        F} \mu_f |\p_x u_f|^2 \mathrm{d}x+ {2}\mu_g \sum_{k=1}^N
      \dfrac{|\dot R_k|^2}{|R_k|}=0.
  \end{aligned}
  \end{equation}
\end{Prop}

\begin{proof}
  First let multiply the Navier--Stokes equation
  \eqref{eq_fluid_NS_1D} by the velocity $u_f$ and integrate over the
  fluid domain $\mathcal F$. Using the mass conservation equation
  \eqref{eq_fluid_mass}, it yields
\begin{equation}
  \label{eq_estim_nrj_1}
  \int_{\mathcal F} \rho_f \left( \p_t u_f + u_f \p_x u_f\right) u_f
  \mathrm{d}x = \int_{\mathcal F} u_f \p_x \Sigma_f   \mathrm{d}x .
\end{equation}
Since the mass conservation~\eqref{eq_fluid_mass} gives
  \begin{equation*}
  \dfrac{\mathrm{d}}{\mathrm{d}t}  \int_{x_k^+}^{x_{k+1}^-}
  \rho_f \dfrac{|u_f|^2}{2} \mathrm{d}x =
   \int_{x_k^+}^{x_{k+1}^-} 
   \rho_f (\p_t u_f + u_f \p_x u_f) u_f \mathrm{d}x,
 \end{equation*}
 one obtains, using an integration by part of  the right-hand side,
 \begin{equation*}
     \dfrac{\textrm{d}}{\textrm{d}t}  \int_{\mathcal F} \rho_f
     \dfrac{|u_f|^2}{2} \textrm{d}x = T_1 -T_2-T_3,
   \end{equation*}
   with
   \begin{align*}
     T_1 &= \sum_{k=0}^N \Sigma_f (x_{k+1}^-) u_f(x_{k+1}^-) -
           \Sigma_f (x_{k}^+) u_f(x_{k}^+),\\
     T_2 &= \int_{\mathcal F} \mu_f |\p_x u_f|^2 \mathrm{d}x ,\\
     T_3 &= -\int_{\mathcal F} \mathrm{p}_f(\rho_f) \p_x u_f  \mathrm{d}x,
   \end{align*}
   where the terms $T_2$ and $T_3$ come from the definition \eqref{eq_fluid_tensor_sigmaf} of the
   stress $\Sigma_f.$ 

   Using the boundary conditions \eqref{eq_BC_uf} and, after,  the
   continuity of the velocities at the droplet interfaces
   \eqref{eq_continuity_velocity}, the term $T_1$ can be rewritten as
\begin{equation*}
  \begin{aligned}
    T_1 &= -\sum_{k=1}^N \big( \Sigma_f (x_k^+) u_f(x_k^+) -
    \Sigma_f(x_k^-)u_f(x_k^-) \big) \\
    &= -\sum_{k=1}^N \dot c_k\left( \Sigma_f (x_k^+)-\Sigma_f
      (x_k^-)\right) + \dot R_k \left( \Sigma_f (x_k^+)+\Sigma_f
      (x_k^-)\right).
  \end{aligned}
\end{equation*}
Finally the droplets motion equations
\eqref{eq_droplet_newton_1D_C}-\eqref{eq_droplet_newton_1D_R} and the
definition of the droplet pressure law \eqref{eq_droplet_pk} yield (whatever the 
value of $R_{ref} >0$):
\begin{equation*}
  \label{eq_eqtim_nrj_T1_3}
  T_1 =  - \dfrac{\mathrm{d}}{\mathrm{d}t}  \left[
    \left( \sum_{k=1}^N m_k \dfrac{|\dot c_k|^2}{2}
      + \dfrac{m_k}{3}\dfrac{|\dot R_k|^2}{2}\right) - {2}
    \sum_{k=1}^N \kappa_k\ln\left( \dfrac{R_k}{R_\mathrm{ref}}\right)\right]
  -{2}\mu_g \sum_{k=1}^N \dfrac{|\dot R_k|^2}{R_k}.
\end{equation*}
We now turn to the term $T_3$. By the definition \eqref{eq_qTL} of the
function $\mathrm{q}_f$, and by the mass conservation equation
\eqref{eq_fluid_mass}, it holds
\begin{equation*}
  \p_t \mathrm{q}_f (\rho_f) + \p_x (\mathrm{q}_f (\rho_f)u_f)=-\mathrm{p}_f(\rho_f)\p_x u_f.
\end{equation*}
Because the fluid domain evolves with the velocity $u_f$, $T_3$ can be
recovered
\begin{equation*}
  \dfrac{\mathrm{d}}{\mathrm{d}t} \int_{\mathcal F} q(\rho_f) \mathrm{d}x=
  -\int_{\mathcal F} \mathrm{p}_f(\rho_f) \p_x u_f\mathrm{d}x= T_3.
\end{equation*}
One deduces the final estimate \eqref{eq_estim_nrj} combining the
terms $T_1$, $T_2$ and $T_3$.
\end{proof}

In the regime of initial data specified in this section, we obtain the
following corollary:
\begin{Cor} \label{corollaire_bound_energy}
  If initial data are constructed as in
  \eqref{eq_initbubble}-\eqref{eq_initbubble2} and satisfy
  \ref{it_bound_mk0}-\ref{it_bound_R0}-\ref{it_bound_F0}, there exists
  a constant $E_0$ depending only on the list of parameters
  \eqref{eq_listparametre} such that any classical solution to
  \eqref{eq_fluid_mass}--\eqref{eq_droplet_tensor_sigmak} on some
  time-interval $[0,T]$ satisfies:
  \begin{multline}
    \label{cor_bound_energy}
    \displaystyle \int_{\mathcal F} \left( \rho_f
      \dfrac{|u_f|^2}{2}+q(\rho_f)\right)\mathrm{d}x
    + \dfrac{1}{2}\sum_{k=1}^N m_k \big(|\dot c_k|^2 + \dfrac 1 3
    |\dot R_k|^2\big) 
    - 2 \sum_{k=1}^N \kappa_k \ln(d_\infty
    N R_k)  \leq E_0,
  \end{multline}
  on $(0,T)$ with, denoting by $\ln_+$ the positive part of the $\ln$:
  \begin{equation}
    \label{cor_bound_diss}
    \int_0^T\Bigg[\bigg(\int_{\mathcal F} \mu_f |\p_x u_f|^2 \mathrm{d}x+
    \mu_g \sum_{k=1}^N \dfrac{|\dot R_k|^2}{R_k}\bigg)\Bigg]\mathrm{d}t\leq E_0 + 2  \max_{[0,T]} \sum_{k=1}^N \kappa_k \ln_{+}(d_{\infty}N R_k).
  \end{equation}
\end{Cor}

\begin{proof}
To obtain these inequalities, we integrate \eqref{eq_estim_nrj} with
$R_{ref} = 1/{d_{\infty} N}$ and remark that all the terms on the
left-hand side are positive but:
\[
  \sum_{k=1}^N \kappa_k \ln(d_{\infty} N R_k).
\] 
We obtain then the inequalities \eqref{cor_bound_energy} and
\eqref{cor_bound_diss} with:
\[
  E_0 := \int_{\mathcal F^0} \left( \rho_f^0 \dfrac{|u_f^0|^2}{2} +
    q_f(\rho_f^0)\right) {\rm d}x + \sum_{k=1}^N
  \left(\dfrac{|\dot{c}_k^0|^2}{2} + \dfrac{|\dot{R}_k^0|^2}{6} \right)
  - 2 \sum_{k=1}^N \kappa_k \ln (R_k^0 d_{\infty}N).
\]
The first term in $E_0$ is clearly controlled by $\|u_f^0\|_{L^2}$ and
$\bar{\rho}_{\infty}.$
As for the second term, the velocity continuity
\eqref{eq_initbubble1}-\eqref{eq_initbubble2} gives
\begin{equation*}
  \label{eq_IC-ck0}
  |\dot c_k^{0}|+  |\dot R_k^{0}| \leq 2\|
  u_f^0\|_{L^\infty(\Omega)}, \quad \forall k=1,\ldots,N .
\end{equation*}
Then, with \ref{it_bound_mk0}, we obtain:
\begin{equation*}
  \label{eq_IC-ck1}
\sum_{k=1}^N m_k \left( |\dot c_k^0|^2 + \dfrac 1 3
    |\dot R_k^0|^2\right)\leq \dfrac{4}{M_{\infty}}\|u_f^0\|^2_{L^\infty(\Omega)},
\end{equation*}
and, with a classical Sobolev embedding, this part is again controlled
by $M_{\infty}$ and $\|u_f^0\|_{H^1_0(\Omega)}.$
Now using the bound \ref{it_bound_R0} on the initial radii, it holds
\begin{equation*}
  \label{eq_IC-Rk1}
  2d_\infty^2\leq R_k^0 N d_\infty \leq \dfrac 1 2,
\end{equation*}
so that
\begin{equation*}
  \label{eq_IC-Rk2}
  - \sum_{k=1}^N \kappa_k \ln(d_\infty NR_k^0)\leq \dfrac{|\ln(2d_{\infty}^2)|}{M_{\infty}}.
\end{equation*}
This concludes the proof.
\end{proof}

We proceed with a second classical regularity estimate: 
\begin{Prop}
  \label{prop_semi_norm_dissipation}
  The following identity holds
  \begin{equation}
    \label{eq_semi_norm_dissipation}
    \begin{aligned}
      \dfrac{\mathrm{d}}{\mathrm{d}t} &\left[\int_{\mathcal F}
        \left(\mu_f \dfrac{|\p_x u_f|^2}{2}-\mathrm{p}_f(\rho_f) \p_x
          u_f\right) \mathrm{d}x + \sum_{k=1}^N \left( \mu_g
          \dfrac{|\dot R_k|^2}{R_k} - {2} \kappa_k \dfrac{\dot
            R_k}{R_k}\right)
      \right]\\
      &+ \int_{\mathcal F} \rho_f|\p_t u_f + u_f \p_x u_f |^2
      \mathrm{d}x+ \sum_{k=1}^N m_k \left(|\ddot c_k|^2 + \dfrac 1 3
        |\ddot R_k|^2
      \right)\\
      & = \int_{\mathcal F} \left( \mathrm{p}_f'(\rho_f)\rho_f |\p_x u_f|^2 -
        \mu_f
        \dfrac{(\p_x u_f)^3}{2}\right) \mathrm{d}x \\
      &+ \sum_{k=1}^N \left({2} \kappa_k \dfrac{|\dot R_k|^2}{R_k^2} -
        \mu_g \dfrac{|\dot R_k|^3}{R_k^2}\right).
    \end{aligned}
  \end{equation}
\end{Prop}

\begin{proof}
  Multiplying the momentum equation \eqref{eq_fluid_NS_1D} by
  $\p_t u_f + u_f \p_x u_f$ and integrating over the fluid domain
  $\mathcal F$ yield
  \begin{equation}
    \label{eq_semi_norm_dissipation2}
    \begin{aligned}
      \int_{\mathcal F} \rho_f|\p_t u_f + u_f \p_x u_f|^2 \mathrm{d}x
      &= \int_{\mathcal F} (\p_t u_f + u_f \p_x u_f) \p_x \Sigma_f
      \mathrm{d}x \\
      &= T_4-T_5,
    \end{aligned}
  \end{equation}
  with
  \begin{align*}
    T_4 &= \sum_{k=0}^N \Sigma_f(x_{k+1}^-)(\p_t u_f + u_f \p_x
          u_f)(x_{k+1}^-) - \Sigma_f(x_{k}^+)(\p_t u_f + u_f \p_x
          u_f)(x_{k}^+),\\
    T_5 &=\int_{\mathcal F} \Sigma_f\p_x (\p_t u_f + u_f
          \p_x u_f) \textrm{d}x .
  \end{align*}
   The boundary term $T_4$ can be simplified by using the interface
  conditions \eqref{eq_continuity_velocity},
  \begin{equation*}
    \dfrac{\textrm{d}}{\textrm{d}t} (\dot c_k \pm \dot R_k) =
    \dfrac{\textrm{d}}{\textrm{d}t} \left( u_f(x_k^\pm)\right) =
    \left( \p_t u_f + u_f \p_x u_f\right)(x_k^\pm).
  \end{equation*}
  Then one obtains
    \begin{equation*}
    T_4 = \sum_{k=0}^N \Sigma_f(x_{k+1}^-) (\ddot c_{k+1} -  \ddot R_{k+1}) -
    \Sigma_f(x_{k}^+) (\ddot c_k + \ddot R_k).
  \end{equation*}
  The boundary conditions \eqref{eq_BC_uf} allow to reorganize the
  sum, and using the droplet equations of motion
  \eqref{eq_droplet_newton_1D_C}--\eqref{eq_droplet_newton_1D_R} and
  the droplet pressure law \eqref{eq_droplet_pk}, we have successively
  \begin{align*}
    T_4 &=- \left\{ \sum_{k=1}^N \Sigma_f(x_{k}^+) (\ddot c_k + \ddot R_k) -
          \Sigma_f(x_{k}^-) (\ddot c_k - \ddot R_k)\right\}\\
        & = - \sum_{k=1}^N \ddot c_k(\Sigma_f(x_k^+) - \Sigma_f(x_k^-))
          + \ddot R_k (\Sigma_f(x_k^+) +\Sigma_f(x_k^-)) \\
        &= - \sum_{k=1}^Nm_k\left( |\ddot c_k|^2 + \dfrac 1 3 |\ddot
          R_k|^2\right)+ {2} \ddot R_k \left( \mu_g \dfrac{\dot R_k}{R_k}-
          \dfrac{\kappa_k}{R_k}\right) \\
        &= - \sum_{k=1}^N\left\{m_k \left( |\ddot c_k|^2 + \dfrac 1 3 |\ddot
          R_k|^2\right)+\dfrac{\mathrm{d}}{\mathrm{d}t} \left[ \mu_g
          \dfrac{|\dot R_k|^2}{R_k} - {2} \kappa_k \dfrac{\dot
          R_k}{R_k}\right]\right\}\\
    &\quad+ \sum_{k=1}^N \left( {2} \kappa_k \dfrac{|\dot R_k|^2}{R_k^2}-
      \mu_g   \dfrac{(\dot R_k)^3}{R_k^2}\right).
  \end{align*}
  We now turn to the volumic term $T_5$.  Developing the term $T_5$
  gives
  \begin{equation}
    T_5 = T_6 + \int_{\mathcal F}\mu_f (\p_x u_f)^3 \mathrm{d}x -T_7
    -\int_{\mathcal F} \mathrm{p}_f(\rho_f) |\p_x u_f|^2 \mathrm{d}x,
  \end{equation}
  with
  \begin{align*}
    T_6 &= \int_{\mathcal F}\mu_f \p_x u_f \left( \p_t (\p_x u_f) +
          u_f \p_x(\p_x u_f)\right) \textrm{d}x,\\
    T_7 &=\int_{\mathcal F}\mathrm{p}_f(\rho_f) \left(\p_t (\p_x u_f)
          + u_f \p_x(\p_x u_f) \right) \textrm{d}x .
  \end{align*}
  These two terms can be handled by classical manipulations, providing
  \begin{align*}
    \dfrac{\textrm{d}}{\textrm{d}t} \left[\int_{\mathcal F}\mu_f \dfrac{|\p_x
    u_f|^2}{2} \textrm{d}x\right] &= T_6 + \int_{\mathcal F}\mu_f \dfrac{( \p_x
                                    u_f)^3}{2}, \\
    \dfrac{\textrm{d}}{\textrm{d}t} \left[\int_{\mathcal
    F}\mathrm{p}_f(\rho_f)\p_x u_f \textrm{d}x\right] 
                                  & = T_7- \int_{\mathcal F}
                                    (\mathrm{p}_f(\rho_f)-\mathrm{p}_f'(\rho_f)\rho_f)|\p_x
                                    u_f|^2 \textrm{d}x. 
  \end{align*}
  As a result, 
    \begin{equation*}
      \begin{aligned}
        T_5 &= \dfrac{\textrm{d}}{\textrm{d}t} \left[ \int_{\mathcal F}
          \mu_f \dfrac{|\p_x u_f|^2}{2} - \mathrm{p}_f(\rho_f) \p_x
          u_f \textrm{d}x\right] \\
        &\quad + \mu_f \int_{\mathcal F}
        \dfrac{(\p_x u_f)^3}{2}\textrm{d}x + \int_{\mathcal F}
        p'(\rho_f) \rho_f |\p_x u_f|^2 \textrm{d}x.
      \end{aligned}
  \end{equation*}
    Finally plugging the expressions of $T_4$ and $T_5$ into
  \eqref{eq_semi_norm_dissipation2}
  gives the expected result.
\end{proof}

In the regime of initial data specified in this section, we obtain the
following corollary:
\begin{Cor} \label{cor_est_H1}
If initial data are constructed as in
\eqref{eq_initbubble}-\eqref{eq_initbubble2} and satisfy
\ref{it_bound_mk0}-\ref{it_bound_R0}-\ref{it_bound_F0}, there exists a
constant $E_1$ depending only on the list of parameters
\eqref{eq_listparametre} such that any classical solution to
\eqref{eq_fluid_mass}--\eqref{eq_droplet_tensor_sigmak} on some
time-interval $[0,T]$ satisfies:
 \begin{equation}
   \label{eq_bound_H1_1}
   \begin{aligned}
   &  \sup\limits_{[0,T]} \left( \int_{\mathcal F} \right.\left.\mu_f \dfrac{|\p_x
         u_f|^2}{2} \mathrm{d}x+ \mu_g\sum_{k=1}^N \dfrac{|\dot
         R_k|^2}{R_k}\right)\\
     &+ \int_0^T \left(\int_{\mathcal F}\rho_f |\p_t u_f + u_f \p_x
       u_f|^2 \mathrm{d}x
     + \sum_{k=1}^N m_k (|\ddot c_k|^2 + |\ddot
     R_k|^2)\right)\\
   &\leq  \sup\limits_{[0,T]} \left[\left( 2\sum_{k=1}^N \kappa_k \dfrac{|\dot R_k|}{R_k}\right)+
   \int_{\mathcal F} \mathrm{p}_f(\rho_f) |\p_x u_f| \mathrm{d}x\right] + \int_0^T \int_{\mathcal F} \mu_f \dfrac{|\p_x u_f|^3}{2}\mathrm{d}x\\
   &+    \int_0^T \sum_{k=1}^N \left( 2\kappa_k \dfrac{|\dot R_k|^2}{R_k^2} + \mu_g \dfrac{|\dot
       R_k|^3}{ R_k^2}\right) + E_1.
 \end{aligned}
\end{equation}
\end{Cor}
\begin{proof}
Integrating identity
 \eqref{eq_semi_norm_dissipation}  given in Proposition
 \ref{prop_semi_norm_dissipation} between $0$ and $t \leq T$,
 rejecting all non-signed term on the right-hand side that we bound
 then by putting absolute values,  it yields:
 \begin{equation*}
   \begin{aligned}
     \left( \int_{\mathcal F} \right.&\left.\mu_f \dfrac{|\p_x
         u_f|^2}{2} \mathrm{d}x+ \mu_g\sum_{k=1}^N \dfrac{|\dot
         R_k|^2}{R_k}\right)\\
     &+ \int_0^t \left(\int_{\mathcal F}\rho_f |\p_t u_f + u_f \p_x
       u_f|^2 \mathrm{d}x
       + \sum_{k=1}^N m_k (|\ddot c_k|^2 + |\ddot
       R_k|^2)\right)\\
     &+   \int_0^T \int_{\mathcal F} \kappa_f \gamma_f \rho_f^{\gamma_f}
     |\p_x u_f|^2 \mathrm{d}x\\
     &\leq   \left[\left( 2\sum_{k=1}^N \kappa_k \dfrac{|\dot R_k|}{R_k}\right)+
       \int_{\mathcal F} \mathrm{p}_f(\rho_f) |\p_x u_f|
       \mathrm{d}x\right] + \int_0^t \int_{\mathcal F} \mu_f
     \dfrac{|\p_x u_f|^3}{2}\mathrm{d}x\\
     &+    \int_0^t \sum_{k=1}^N \left(2 \kappa_k \dfrac{|\dot R_k|^2}{R_k^2} + \mu_g \dfrac{|\dot
         R_k|^3}{ R_k^2}\right) +\int_{\mathcal F^0} \mu_f \dfrac{|\p_x
         u_f^0|^2}{2} \mathrm{d}x+ \sum_{k=1}^N   \left(  \mu_g \dfrac{|\dot
         R_k^0|^2}{R_k^0} {\color{red}+ 2\kappa_k \dfrac{|\dot R^0_k|}{R^0_k}} \right).
   \end{aligned}
 \end{equation*}
 
 To obtain the expected result, it remains to drop the last term in the
 left-hand side which is positive and to bound the last
 term on the right-hand side by a constant $E_1$ with the expected
 dependencies.
 For this, we note that the first integral
 in this last term clearly depends on $\|u_f^0\|_{H^1_0(\Omega)}.$
 Concerning the first term in the sum, the
 continuity of the velocity field \eqref{eq_initbubble2}
 rewrites for any $k:$
 \[
   \dot R_k^{0}=\dfrac 1 2 \displaystyle \int_{B_k^0} \p_x u_f^{0}(s) \textrm{d}s,
 \] 
 so that 
 \begin{equation*}
   \label{eq_IC-dotRk}
   |\dot R_k^{0}|\leq \dfrac 1 2 \sqrt{R_k^{0}} \left( \int_{B_k^0}
     |\p_x u_f^{0}(s)|^2 \textrm{d}s\right)^{1/2}.
 \end{equation*}
 As a consequence it holds
 \begin{equation*}
   \label{eq_IC-dotRk2}
   \sum_{k=1}^N \dfrac{|\dot R_k^{0}|^2}{R_k^{0}}\leq \dfrac 1 2 \int_{\cup B_k^0}
   |\p_x u_f^{0}(s)|^2 \textrm{d}s\leq \|u_f^0\|_{H^1(\Omega)}^2.
 \end{equation*}
 {
 Finally, the last term in the sum is bounded by using that $\kappa_k$ scales like $1/N.$ Indeed, applying \ref{it_bound_mk0} with \ref{it_bound_R0} we have:
\[ \dfrac{\kappa_k}{\sqrt{R_k^0}} \leq \dfrac{M_{\infty}}{\sqrt{2d_{\infty}}} \dfrac{1}{\sqrt{N}} \quad \forall \, k=1,\ldots,N,
\]
and then, with the above bound on $|\dot{R}_k^0|/\sqrt{R_k^0},$ we obtain: 
 \[
 \sum_{k=1}^N \kappa_k \dfrac{|\dot{R}_k^0|}{R_k^{0}} \leq \dfrac{M_{\infty}}{\sqrt{8d_{\infty}}} \left( \int_{\cup B_k^0}|\partial_x u_f^0|^2\right)^{\frac 12}.
 \]
}
 This ends the proof.
\end{proof}

\subsection{Extended stress-tensor estimates}
\label{sec_stress-estimates}

In order to obtain regularity estimates on the fluid velocity field, a classical
way is to use the stress tensor. However $\Sigma_f$ is only defined
on the fluid domain $\mathcal{F}$, so that estimates on this stress
tensor depend on the geometric properties of $\mathcal F$, in
particular the number of bubbles.
In order to remove this
dependency, we define new stress tensors for the fluid and for
the gas phase, extended to the full domain $\Omega$:
\begin{equation}
  \label{eq_tSigmaf}
  \tSigma_f =
  \begin{cases}
    \Sigma_f , & \text{in } \mathcal F,\\
    \dfrac{\Sigma_f(x_k^-)+\Sigma_f(x_k^+)}{2} -
    \dfrac{\Sigma_f(x_k^-)-\Sigma_f(x_k^+)}{2R_k}(x-c_k), & \text{in }
    B_k, \;  k=1,\dots,N,
  \end{cases}
\end{equation}
and
\begin{equation}
  \label{eq_tSigmag}
  \tSigma_g =
  \begin{cases}
    \Sigma_k, & \text{in } B_k, \; k=1,\dots, N,\\
    \Sigma_N, & \text{in }\mathcal F_N,\\
    \Sigma_0, & \text{in }\mathcal F_0,\\
    \Sigma_k + \dfrac{\Sigma_{k+1}-
      \Sigma_k}{x_{k+1}^--x_k^+}(x-x_k^+), & \text{in } \mathcal
    F_k, \; k=1,\dots, N-1.
  \end{cases}
\end{equation}
Observe that these two stress tensors are continuous at each interface
$x_k^\pm$. We analyze here the properties of these extensions, when
$\Sigma_f$ obeys further the continuity
properties adapted from
\eqref{eq_droplet_newton_1D_C}-\eqref{eq_droplet_newton_1D_R}-\eqref{eq_droplet_tensor_sigmak}.
Namely:
\begin{align}
  \label{eq_droplet_newton_1D_C_bis}
  &m_k \ddot c_k = \Sigma_f(x_k^+) - \Sigma_f(x_k^-) , \\
  \label{eq_droplet_newton_1D_R_bis}
  &\displaystyle \frac{m_k}{3} \ddot R_k = \Sigma_f(x_k^-) +
  \Sigma_f(x_k^+)  - 2 \Sigma_k , \\
   \label{eq_droplet_tensor_sigmak_bis}
   &\displaystyle \Sigma_k = \mu_g \frac{\dot R_k}{R_k} - {\rm
     p}_g(\rho_k) - F_s/2,
 \end{align}
In the stationary analysis of this subsection, these
latter identities may stand for definitions
of $\ddot{c}_k$ and $\ddot{R}_k.$ These quantities will be related to
the dynamical problem afterwards.

\begin{Prop}
  \label{prop_estima_sigmaTf}
 Assume that $\Sigma_f \in H^1(\mathcal F)$ satisfies \eqref{eq_droplet_newton_1D_C_bis}-\eqref{eq_droplet_newton_1D_R_bis}
  with $\Sigma_k$ defined by \eqref{eq_droplet_tensor_sigmak_bis}.
  Then $\tSigma_f\in H^1(\Omega)$ and there exists a constant $C_0>0$ such
  that
  \begin{equation}
    \label{eq_stress_fluid2_H1}
    \begin{aligned}
      \|\tSigma_f\|_{H^1(\Omega)} \leq C_0 &\left[
        \|\Sigma_f\|^2_{H^1(\mathcal F)}  + \sum_{k=1}^N (m_k)^2
        \left( |\ddot{R}_{k}|^2+
         \dfrac{|\ddot{c}_k|^2}{R_k}\right)\right.\\
      & \left.+ \sum_{k=1}^N \left( \mu_g^2
          \dfrac{|\dot{R}_k|^2}{R_k} +
          \dfrac{\kappa_k^2}{R_k}\right) \right]^{\frac 12}.
    \end{aligned}
\end{equation}
\end{Prop}

\begin{proof}
  By continuity of $\tSigma_f$ at the interfaces,
    \begin{equation*}
      \|\tilde{\Sigma}_f\|_{H^1(\Omega)}^2 =
      \|{\Sigma}_f\|_{H^1(\mathcal F)}^2 + \sum_{k=1}^N \|\tilde{\Sigma}_f\|_{H^1(B_k)}^2.
  \end{equation*}
  We just have to study the $H^1$ norm of $\tSigma_f$ on a bubble $B_k$.
  The $L^2$ norm of $\Sigma_f$ can be
  bounded as follows:
    \begin{equation*}
    \begin{aligned}
      \|\tilde\Sigma_f\|_{L^2(B_k)}^2 & = \int_{B_k}
      \left|\dfrac{\Sigma_f(x_k^-)+\Sigma_f(x_k^+)}{2}\right|^2 +
      \left|\dfrac{\Sigma_f(x_k^-)-\Sigma_f(x_k^+)}{2R_k}\right|^2|x-c_k|^2\textrm{d}x\\
      & = 
      \left|\dfrac{\Sigma_f(x_k^-)+\Sigma_f(x_k^+)}{2}\right|^22R_k
      +
      \left|\dfrac{\Sigma_f(x_k^-)-\Sigma_f(x_k^+)}{2R_k}\right|^2\dfrac{2R_k^3}{3}\\
      & = \left|{\Sigma_f(x_k^-)+\Sigma_f(x_k^+)}\right|^2\dfrac{R_k}{2}
      +\left|{\Sigma_f(x_k^-)-\Sigma_f(x_k^+)}\right|^2\dfrac{R_k}{6}.
    \end{aligned}
  \end{equation*}
  On the other hand,
  \begin{equation*}
    \begin{aligned}
      \|\p_x \tilde\Sigma_f\|_{L^2(B_k)}^2 &= \int_{B_k}
      \left|\dfrac{\Sigma_f(x_k^-)-\Sigma_f(x_k^+)}{2R_k}\right|^2\textrm{d}x\\
      &= \left|\dfrac{\Sigma_f(x_k^-)-\Sigma_f(x_k^+)}{2R_k}\right|^2
      2 R_k= \dfrac{|\Sigma_f(x_k^-)-\Sigma_f(x_k^+)|^2}{2R_k}.
    \end{aligned}
  \end{equation*}
  We now gather the two estimates, and obtain
  \begin{align*}
      \|\tilde \Sigma_f\|_{H^1(B_k)}^2 &=
      \dfrac{|\Sigma_f(x_k^-)-\Sigma_f(x_k^+)|^2}{2R_k} + 
      \left|{\Sigma_f(x_k^-)+\Sigma_f(x_k^+)}\right|^2\dfrac{R_k}{2}\\
      &+\left|{\Sigma_f(x_k^-)-\Sigma_f(x_k^+)}\right|^2\dfrac{R_k}{6}.
  \end{align*}
Using the equations of motion of the droplets
  \eqref{eq_droplet_newton_1D_C} and the definition
  \eqref{eq_droplet_tensor_sigmak} of the stress
  tensor $\Sigma_k$, one gets
    \begin{align*}
      \|\tilde \Sigma_f\|_{H^1(B_k)}^2 &\leq
      m_k^2 |\ddot c_k|^2 \left(\dfrac{1}{2R_k} + \dfrac{R_k}{6}\right)
      + \left( \dfrac{m_k}{3}\ddot R_k + 2 \left( \mu_g \dfrac{\dot
            R_k}{R_k}- \dfrac{\kappa_k}{R_k}\right)\right)^2 \dfrac{R_k}{2} \\
      &\leq m_k^2 |\ddot c_k|^2 \left(\dfrac{1}{2R_k} + \dfrac{R_k}{6}\right)
      + \dfrac 2 9 m_k^2 |\ddot R_k|^2R_k + 8 \mu_g^2 \dfrac{|\dot
        R_k|^2}{R_k}+ 8 \dfrac{\kappa_k^2}{R_k}
    \end{align*}
    Finally, this gives the estimate
    \begin{multline*}
      \|\tilde{\Sigma}_f\|_{H^1(\Omega)}^2 \leq {8}\left[
        \|\Sigma_f\|^2_{H^1(\mathcal F)}  + \sum_{k=1}^N {(m_k)^2}
        \left( |\ddot{R}_{k}|^2{ R_k}+
      |\ddot{c}_k|^2\left( \dfrac{1}{R_k}
       +R_k \right)\right)\right.\\
\left.+ \sum_{k=1}^N \left( \mu_g^2
          \dfrac{|\dot{R}_k|^2}{R_k} +
          \dfrac{\kappa_k^2}{R_k}\right) \right]^{\frac 12},
    \end{multline*}
    which leads to the desired result since $R_k< 1$.
\end{proof}

From the above inequality we deduce the following $L^{\infty}$-bound
in case $\Sigma_{f}$ is a viscous stress tensor:

\begin{Prop}
  \label{cor_Linf_dxuf}
  Assume that $\Sigma_f \in H^1(\mathcal F)$ satisfies \eqref{eq_droplet_newton_1D_C_bis}-\eqref{eq_droplet_newton_1D_R_bis}
  with $\Sigma_k$ defined by \eqref{eq_droplet_tensor_sigmak_bis}. Assume further that $\Sigma_f$ 
  is related to $(\rho_f,u_f) \in H^1(\mathcal F) \times H^2(\mathcal F)$ \emph{via} \eqref{eq_fluid_tensor_sigmaf}.
  Then, there exists $C_1>0$ such that
  \begin{equation}
    \label{eq_Linf_dxuf}
    \|\p_x u_f\|_{L^\infty (\mathcal F)} \leq \dfrac{C_1}{\mu_f}\left(
      \|\tilde \Sigma_f\|_{H^1(\Omega)} +
      \|\mathrm{p}_f(\rho_f)\|_{L^\infty(\mathcal F)}\right).
  \end{equation}
\end{Prop}

\begin{proof}
  In the fluid domain, the stress tensor writes $\Sigma_f = \mu_f \p_x
  u_f -\mathrm{p}_f$, which gives
  \begin{equation*}
    \label{eq: Linf_dxuf_1}
    \p_x u_f = \dfrac{1}\mu_f(\Sigma_f - \mathrm{p}_f(\rho_f)).
  \end{equation*}
  Hence one has
  \begin{equation*}
    \label{eq:dxuf_1}
    \|\p_x u_f\|_{L^\infty (\mathcal F)} \leq
    \dfrac{1}\mu_f(\|\Sigma_f\|_{L^\infty(\mathcal F)}
    +\|\mathrm{p}_f(\rho_f)\|_{L^\infty(\mathcal F)} ).
  \end{equation*}
  The definition of global tensor $\tilde \Sigma_f$ gives then
  \begin{equation*}
    \label{eq: Linf_sigmaf}
    \|\Sigma_f\|_{L^\infty(\mathcal F)} \leq \|\tilde\Sigma_f\|_{L^\infty(\Omega)}.
  \end{equation*}
  The $H^1(\Omega)\subset
  L^\infty(\Omega)$ embedding allows to conclude the proof.
\end{proof}

One can note here the gain of working with an extended stress tensor.
Indeed, the constant $C_1$ we obtain in the previous
proposition is independent of the position of the particles and their
radius.
This would not be \emph{a priori} the case if we wanted
to control $\partial_ xu$ by $\Sigma_f$ only.  
Nevertheless, in \eqref{eq_stress_fluid2_H1} we introduced on the
right-hand side negative powers of $R_k$ that we shall control
independently. To this end, we performed a symmetric construction with
the bubble stress-tensor $\Sigma_g$
and we provide  now a corresponding proposition:

\begin{Prop}
  \label{prop_tSigmag}
  Assume that $\Sigma_{f}$ and $(\Sigma_k)_{k=1,\ldots,N}$ are related {\em via}
  \eqref{eq_droplet_tensor_sigmak_bis}. Then
  $\tSigma_g\in H^1(\Omega)$ and there exists a constant $C_2>0$ such
  that
  \begin{equation}
    \label{eq_max_k_rkp}
    \begin{aligned}
      \|\tilde \Sigma_g\|_{H^1(\Omega)} \leq C_2 \left[ \| \tilde
        \Sigma_f\|_{H^1(\Omega)}^2 + \dfrac{1}{\min\limits_{k\in\{0,\dots,N\}}
          |\mathcal{F}_k|} \sum_{k=1}^N (m_k)^2 (|\ddot R_k|^2 + |\ddot
        c_k|^2) \right]^{1/2}.
    \end{aligned}
  \end{equation}
\end{Prop}

\begin{proof}
  By straightforward calculations, the definition of $\tilde \Sigma_g$ yields
  \begin{equation*}
    \label{eq_estim_sigmaTg_5}
    \begin{aligned}
      \|\tilde \Sigma_g\|_{H^1(\Omega)}^2 &\leq \sum_{k=1}^N 2 R_k
      |\Sigma_k|^2
      + |\Sigma_0|^2 |x_1^- -x_0^+|+ |\Sigma_N|^2 |x_{N+1}^- -x_N^+|\\
      &\quad+ \sum_{k=1}^{N-1}\bigg(
      \dfrac{|\Sigma_{k+1}-\Sigma_k|^2}{|x_{k+1}^- -x_k^+|}+ 2
      |\Sigma_k|^2 |x_{k+1}^--x_k^+| \\
      & \qquad \qquad + \dfrac 2 3
      |\Sigma_{k+1}-\Sigma_k|^2 |x_{k+1}^--x_k^+|\bigg) \\
      &\leq C \left( \sum_{k=1}^N 
        R_k |\Sigma_k|^2 + \sum_{k=0}^N |\mathcal F_{k}||\Sigma_k|^2 +
        \sum_{k=1}^N  \dfrac{|\Sigma_{k+1}-\Sigma_k|^2}{|x_{k+1}^-
          -x_k^+|}\right)
    \end{aligned}
  \end{equation*}
  where $C$ is a positive constant, since the length of the bubbles
  and of the fluid parts are bounded. Summing equations
  \eqref{eq_droplet_newton_1D_C} and \eqref{eq_droplet_newton_1D_R}
  leads to
  \begin{equation}
    \label{eq_estim_sigmaTg_6}
    \Sigma_k = \Sigma_f(x_k^+) -  \dfrac{m_k}{2} \left(\ddot c_k +
      \dfrac{\ddot R_k}{3} \right).
  \end{equation}
  We deduce the following estimates, with some constant $C'>0$,
  \begin{align*}
    |\Sigma_k | &\leq \|\tilde \Sigma_f\|_{L^\infty(B_k)}+ {m_k} C' \left(|\ddot c_k| +
                  |\ddot R_k| \right), \\
    |\Sigma_{k+1}- \Sigma_k| &\leq |x_{k+1}^+-x_k^+|^{1/2}\|\p_x
                               \tilde \Sigma_f\|_{L^2(x_k^+,x_{k+1}^+)} \\
                &\quad + \dfrac{m_k}{2}
                  \big(|\ddot R_k|+ |\ddot c_k|\big) + \dfrac{m_{k+1}}{2}
                  \big(|\ddot R_{k+1}|+|\ddot c_{k+1}|\big).
  \end{align*}
  One can now go back to the estimate on $\tilde\Sigma_g$. Noting the relation:
  \[
    \sum_{k=0}^N |\mathcal F_k| + \sum_{k=1}^N 2R_k = |\Omega|,
  \]
  the  embedding $H^1(\Omega)\subset L^\infty(\Omega)$ implies the expected result.
\end{proof}

As for the fluid stress tensor, we deduce from the previous computation a 
control on the $(\Sigma_k)_{k=1,\ldots,N}$ by applying again the
embedding $H^1(\Omega) \subset L^{\infty}(\Omega)$:

\begin{Cor}
  \label{cor_est_sigmag}
  Under the same assumptions as in {\bf Proposition \ref{prop_tSigmag}}, there holds:
  \begin{equation}
    \label{eq_estim_sigmag}
    \begin{aligned}
      \max_{k=1,\dots, N-1} \bigg | \mu_g \dfrac{\dot
        R_k}{R_k}-\dfrac{\kappa_k}{R_k}\bigg| \leq C_2 \bigg[ &
      \|  \tilde \Sigma_f\|_{H^1(\Omega)}^2 \\
      &+ \dfrac{1}{\min\limits_{k\in\{0,\dots,N\}} |\mathcal{F}_k|}
      \sum_{k=1}^N (m_k)^2 \big(|\ddot R_k|^2 + |\ddot c_k|^2\big)
      \bigg]^{1/2}.
    \end{aligned}
  \end{equation}
\end{Cor}
This latter corollary shall enable to control the radius of the bubble
from below, preventing from collapse.

\subsection{Proof of Theorem \ref{thm_existence}}
\label{sec_cauchy}
We combine now the computations of the previous section to construct a
solution on a time-interval independent of the number $N$ of bubbles.
For this, we show that the following bounds can be continued: 
\begin{enumerate}[label=(\subscript{Q}{{\arabic*}})]
\item\label{it_bound_R}
  $d_\infty\leq N R_k \leq (d_\infty)^{-1}$,
  $k=1,\ldots,N$,\\[-8pt]
\item \label{it_bound_F}
  $d_\infty\leq N |\mathcal{F}_k| \leq (d_\infty)^{-1}$,
  $k=1,\ldots,N$,\\[-8pt]
\item\label{it_bound_rho}  $\underline{\rho}_\infty\leq \rho_f\leq
  \bar{\rho}_\infty$  on $\mathcal{F}(t)$
\end{enumerate} 
and, introducing a sufficiently large $K >0:$
\begin{enumerate}[label=(\subscript{Q}{{\arabic*}})]
  \setcounter{enumi}{3}
\item\label{it_bound_H1} $\bigg[ \displaystyle \int_{\mathcal F} \mu_f  \dfrac{|\p_x
    u_f|^2}{2}\mathrm{d}x
  + \mu_g\sum_{k=1}^N \dfrac{|\dot R_k|^2}{R_k} \bigg] \leq K$,
\item\label{it_H1_bound_tensors}
  $\displaystyle \int_0^t\Bigg[ \|\tilde \Sigma_f\|_{H^1(\Omega)}^2
  + \|\tilde \Sigma_g\|_{H^1(\Omega)}^2 +
  \sum_{k=1}^N {m_k} \big(|\ddot R_k|^2 + |\ddot
  c_k|^2 \big) \Bigg]\mathrm{d}s\leq K$.
\end{enumerate}
 We keep the convention here that tildas represent
 extended stress tensors as constructed in the previous subsection.
 We prove that, if $K$ is chosen sufficiently large wrt the list of
 parameters \eqref{eq_listparametre}, then we have such estimates on a
 time interval $(0,T)$ that depends only on the same list of parameters
 \eqref{eq_listparametre} (possibly {\em via} $K$).

\medskip

Technically, we apply a continuation argument based on the {\em a
  priori} assumption that the solution exists.
The precise statement is the following proposition in which we
denote $(\mathscr{Q}_i)_{i=1,\ldots,5}$ the estimates
corresponding to the above $(Q_i)_{i=1,\ldots,5}$ where large
inequalities are replaced with strict inequalities.
Tacitly, all constants that are introduced in the following
proposition may depend on the list of parameters
\eqref{eq_listparametre}.

\begin{Prop}
  \label{prop_estimates}
 There exists $K_\infty>0$ such that, for any $K >K_{\infty}$ there
 exists  $T_{\infty}[K] >0$ for which the following statement holds:
 if $T \leq T_{\infty}[K]$ and
 $((\rho_f,u_f),(c_k,R_k)_{k=1,\ldots,N})$ is a classical solution to
 \eqref{eq_fluid_mass}-\eqref{eq_droplet_tensor_sigmak} on $(0,T)$
 satisfying \ref{it_bound_R}-\ref{it_H1_bound_tensors}
  then it satisfies also ($\mathscr Q_1$)-($\mathscr Q_5$).
 \end{Prop}  

The proof of {\bf Proposition \ref{prop_estimates}} is the content of
Appendix~\ref{app_proof-prop-refpr}. We explain here how it implies
{\bf Theorem \ref{thm_existence}}. For this, given $K >0$ we
introduce:
\[
  \mathcal I := \{T \in (0,\infty) \text{ s.t. the unique classical solution
    exists on $(0,T)$ }\text{and satisfies $(Q_1)$--$(Q_5)$}  \}.
\]
Firstly, thanks to the local-in-time existence result, there exists $T_0$
depending on $N$ such that we have a classical solution on $(0,T_0).$
Indeed, for such a solution the radius $R_k$ and $c_k$ are continuous in time. 
Since we assume initially \ref{it_bound_R0}-\ref{it_bound_F0} (resp. \eqref{it_bound_rho0}) we have
that, up to restrict $T_0,$ this solution satisfies
\ref{it_bound_R}-\ref{it_bound_F} (resp.  \ref{it_bound_rho}) on $[0,T_0].$
Similarly, we remark that the quantities on the left-hand side of
\ref{it_bound_H1}-\ref{it_H1_bound_tensors} are continuous
time-dependent functions of the classical solution. Since 
the left-hand side of \ref{it_bound_H1} is controlled 
initially by $\|u_f^0\|_{H^1_0(\Omega)}$ and parameters involved in \eqref{eq_listparametre}
(see the proof of {\bf Corollary \ref{cor_est_H1}}), there exists $K_0$ sufficiently large depending only on
the list of parameters \eqref{eq_listparametre} such that we can
enforce  \ref{it_bound_H1}-\ref{it_H1_bound_tensors} on $[0,T_0]$ also
whatever the value of  $K >K_0$.

\medskip 

Let  fix now $K = \max(K_0,K_{\infty})$ with $K_{\infty}$ given by
{\bf Proposition \ref{prop_estimates}} and denote $T_{\infty} =
T_{\infty}[K].$ By the previous arguments, we have that $[0,T_0]
\subset \mathcal I.$ We show now that
$[0,T_{\infty}] \subset \mathcal I$ which shall end the proof. By restriction,
$\mathcal I \cap [0,T_{\infty}]$ is a closed subinterval of $[0,T_{\infty}]$ containing
$[0,T_0].$ Let us prove that $\mathcal I \cap [0,T_{\infty}]$ is open
(in $[0,T_{\infty}]$).
Indeed, assume $[0,T]$ is a strict subinterval of  $[0,T_{\infty}]$ in
$\mathcal I,$ then we can apply {\bf Proposition \ref{prop_estimates}}
and the solution satisfies ($\mathscr Q_1$)-($\mathscr Q_5$)
on $[0,T].$ It remains to show that we can continue the solution beyond $[0,T].$
The inequalities  ($\mathscr Q_1$)-($\mathscr Q_5$)  being strict,
the large inequalities \ref{it_bound_R}-\ref{it_H1_bound_tensors}
shall be satisfied on a slightly longer interval by continuity. To
extend the solution, we note that \ref{it_bound_R}-\ref{it_bound_F} (resp.  \ref{it_bound_rho}) 
entail "a minimum distance between" and "a minimum radius of" bubbles
(resp.   strictly positive distance to vacuum)
on $[0,T]$. Inequality \ref{it_bound_H1} also ensures a (uniform)
bound from above for $\|u_f\|_{H^1(\mathcal F)}$ on $[0,T].$ 
By Proposition \ref{prop_rhoH1} of Appendix
\ref{sec_rhoH1}  we have also a uniform bound for $\|\rho_f\|_{H^1(\mathcal F)}$
(up to take $T_{\infty}$ smaller).
We can then apply the local-in-time existence result with initial data
$((\rho_f(T',\cdot),u_f(T',\cdot)),(c_k(T'),R_k(T'))_{k=1,\ldots,N})$
for $T'$
arbitrary close to $T.$ This yields a solution on some time-interval
$\Delta T$ (independent of $T',$ given the uniform bound above).
By concatenation, we obtain  a solution on $(0,T' + \Delta T)$ where
$T'+\Delta T > T$ for a well-chosen $T'.$

\medskip

To conclude this section, we mention that the proof above entails that
we have the following corollary to {\bf Theorem \ref{thm_existence}}:

\begin{Cor}
  \label{cor_existence}
  The unique classical solution to
  \eqref{eq_fluid_mass}-\eqref{eq_droplet_tensor_sigmak} on
  $[0,T_{\infty}]$ satisfies the bounds
  \ref{it_bound_R}-\ref{it_bound_F} (resp.  \ref{it_bound_rho}) with
  $d_{\infty}$ corresponding to \ref{it_bound_R0}-\ref{it_bound_F0}
  (res. $\underline{\rho}_{\infty}$, $\bar{\rho}_{\infty}$ corresponding to \eqref{it_bound_rho0})  and \ref{it_bound_H1}-\ref{it_H1_bound_tensors} with $K_{\infty}$
  depending on the list of parameters \eqref{eq_listparametre}.
\end{Cor}

\section{Construction of macroscopic unknowns}
\label{sec_homog-probl}

In this section, we detail the construction of  the unknowns for the
macroscopic model starting from a sequence of  solutions
to the microscopic model with increasing number of gas bubbles. The full justification of the system
\eqref{eq_macromodel}--\eqref{eq_bT}  is postponed to the next section.
{
From now on,  we fix initial data $(\brho_f^0,\brho_g^0,\bu^0,\balpha_f^{0},\balpha_g^{0},\bar{f}_g^{0})$ for the macroscopic model. 
All these quantities are $H^1(\Omega)$ functions. We assume further that they fulfill conditions \eqref{eq_cond0}-\eqref{eq_cond1}-\eqref{eq_cond2}. 
}

\medskip


{
The framework identified in the previous section must be adapted for homogenization purpose.  For instance, given a $N$-bubble solution the gas unknowns at-hand are {\em a priori} the discrete set of center/radius/mass  $(c_k,R_k,m_k)_{k=1,\ldots,N}.$ From them,  we can reconstruct a (functional) density and a covolume by defining for instance: 
\begin{equation} \label{eq_def_dens&cv}
f_g^{(N)}  := \sum_{k=1}^N \dfrac{1}{2NR_k}\mathds{1}_{B_k}  \qquad
\rho_g^{(N)} := \sum_{k=1}^N\dfrac{m_k}{2R_k}  \mathds{1}_{B_k}.
\end{equation}
However,  these reconstructed functions experience $O(1)$ jumps through bubble/fluid interfaces and might not have sufficient regularity to perform the homogenization process.  To gain regularity, we shall propagate an initial regularity through a well-chosen evolution equation (which extends the one satisfied by $f_g^{(N)},\rho_{g}^{(N)}$ on the $B_k$).  However, this requires to be able to construct regular initial covolume and density (with uniform bounds in terms of $N$).  This is obtained  with the following proposition:
\begin{Prop} \label{prop_cstr_id}
Under the assumption that the initial data fulfill the conditions \eqref{eq_cond0}-\eqref{eq_cond1}-\eqref{eq_cond2},  there exist sequences of initial bubble 
center/radii  $((c_k^{(N),0},R_k^{(N),0})_{k=1,\ldots,N})_{N\in \mathbb N}$
and masses $(m_k^{(N)})_{k=1,\ldots,N}$  so that:
\begin{itemize}
\item[i)]  \ref{it_bound_mk0}-\ref{it_bound_R0}-\ref{it_bound_F0} are satisfied with $M_{\infty}$ and $d_{\infty}$ independent of $N$,
\item[ii)] there exist $H^1(\Omega)$ extensions  $(\tilde{f}_g^{(N),0},\tilde{\rho}_g^{(N),0})$ of the associated reconstructed covolumes and densities  such that:\\[-8pt]
\begin{itemize}
\item[$\bullet$] $(\tilde{f}_g^{(N),0},\tilde{\rho}_g^{(N),0})$ is bounded in $H^1(\Omega)$\\[-6pt]
\item[$\bullet$] for arbitrary $\beta \in C^1([0,\infty)\times[0,\infty))$ there holds: 
\[
\beta(\tilde{\rho}_g^{(N),0},\tilde{f}_g^{(N),0})\mathds 1_{\Omega \setminus \bar{\mathcal F}^{(N),0}} \rightharpoonup \bar{\alpha}_g^0 \beta(\bar{\rho}_g^0,\bar{f}_{g}^{0}) \text{  in $\mathcal D'(\Omega).$}
\]
\end{itemize}

\end{itemize}
\end{Prop}
\begin{proof}
Up to a localizing argument, we give a proof in the case:
\[
(1-\alpha_{min}) \|\bar{f}_g^0\|_{L^{\infty}(\Omega)} < f_{min} :=  \inf_{\Omega} \bar{f}_g^0.
\]
To construct our gas bubble, we note that $\bar{\alpha}_g^0\bar{f}_g^{0}$ is a probability density on $\Omega.$ Then, we might construct the
associated cumulative distribution function:
\[
F_g(x) = \int_{-1}^{x} \balpha_g^0 (x) \bar{f}_g^{0}(x){\rm d}x. 
\]
With assumptions \eqref{eq_cond0}-\eqref{eq_cond1}-\eqref{eq_cond2}, this is a $C^1$ one-to-one mapping $\bar{\Omega} \to [0,1]$ with $F'_g \geq \alpha_{min} f_{min}$ on $\Omega.$ 
 We set then:
\begin{equation} \label{eq_calculid}
c_k^0 := F_g^{-1} \left( \dfrac{k}{N+1}\right), 
\quad 
R_k^0 := \dfrac{1}{2N} [\bar{f}_g(c^0_k)]^{-1}
\quad 
m_k :=  2R_k^0 \bar{\rho}_g^0 (c^0_k) 
\quad 
\text{for $k=1,\ldots,N.$}    
\end{equation}
Considering the bounds from above and from below for $F'_g,$ we obtain that:
\[
\dfrac{1}{N+1} \dfrac{1}{(1-\alpha_{min}) \|\bar{f}_g^0\|_{L^{\infty}(\Omega)}} \leq c_{k+1}^0 - c_{k}^0 \leq \dfrac{1}{N+1} \dfrac{1}{\alpha_{min} f_{min}}
\] 
while
\[
\dfrac{1}{2N \|\bar{f}_g^0\|_{L^{\infty}(\Omega)}} \leq  R_k^0 \leq \dfrac{1}{2N} \dfrac{1}{f_{min}}. 
\]
In particular 
\begin{align*}
|\mathcal F_{k}^0| & = (c_{k+1}^0 - R_{k+1}^0 ) - (c_k^0 + R_k^0) \geq \dfrac{1}{N} \left(\dfrac{N/(N+1)}{(1- \alpha_{min})\|\bar{f}_g^0\|_{L^{\infty}(\Omega)} }-\dfrac{1}{f_{min}}.\right) \\
	& \leq c_{k+1}^0 - c_k^0  \leq \dfrac{1}{N} \dfrac{1}{\alpha_{min} f_{min}}
\end{align*}
where $N/((N+1){(1- \alpha_{min}))\|\bar{f}_g^0\|_{L^{\infty}(\Omega)} }-{1}/{f_{min}} >0$
by \eqref{eq_cond1} for $N$ large. 
Finally,  we have:
\[
\dfrac{1}{N} \dfrac{\rho_{min}}{\|\bar{f}_g^0\|_{L^{\infty}(\Omega)}}\leq m_k \leq \dfrac{1}{N} \dfrac{\|\bar{\rho}_g^0\|_{L^{\infty}(\Omega)}}{f_{min}} .
\]
Item i) is  satisfied.

For item ii),  we remark that the reconstructed densities and covolumes read:
\[
f_g^{(N),0}  := \sum_{k=1}^N \dfrac{1}{2NR_k^0}\mathds{1}_{B_k^0}  \qquad
\rho_g^{(N),0} := \sum_{k=1}^N\dfrac{m_k}{2R_k^0}  \mathds{1}_{B_k^0}.
\]
We recall that we denote $B_k^0 = (x_k^{-},x_k^{+})$ where
$x_k^{\pm} = c_k^0 \pm R_k^0$ (and $x_0^{+} = -1,$ $x_{N+1}^{-}$ =1).
At this point, we note that by item i),  we have:
\[
\min_{k \in \{0,\ldots,N\}} |x_{k+1}^{-} - x_{k}^+| \geq \dfrac{1}{2d_{\infty}N}.
\]
Consequently, for $k=2,\ldots,N-1.$ we can construct a piecewise affine function $\psi_k^0$ wich satisfies $\psi_{k}^0 = 1$ on $B_k^0,$ that vanishes in $x_{k+1}^{-}$
and $x_{k-1}^{+}$ and further away from $B_k^0.$ For $k=1$ and $k=N$
we define similarly $\psi_{1}^{0}$ and $\psi_{N}^0$ up to the condition
that $\psi_1^{0}$ is  constant equal to $1$ between $-1$ and $B_1^0$ (resp. $\psi_N^{0}$ is constant equal to $1$ between $B_N^0$ and $1$).
Then, we set:
\[
\tilde{f}_g^{(N),0}  := \sum_{k=1}^N \dfrac{1}{2NR_k^0}\psi_k^0  \qquad
\tilde{\rho}_g^{(N),0} := \sum_{k=1}^N\dfrac{m_k}{2R_k^0}  \psi_k^0.
\]
By standard computations, we have for instance:
\begin{align*}
\|\tilde{f}_{g}^{(N),0}\|^2_{L^2(\Omega)} & \leq  \sum_{k=1}^N \dfrac{1}{N^2 |R_k^0|^2} \|\psi_{k}^{0}\|_{L^2(\Omega)}^2 \\
& \lesssim \dfrac{1}{N} \sum_{k=1}^N \dfrac{1}{N^2|R_k^0|^2} \lesssim \|\bar{f}_g^0\|^2_{L^{\infty}(\Omega)}
\end{align*}
where the first inequality on the second line involves a constant depending on $d_{\infty}.$ We also derive using that $\psi_{k+1} = 1- \psi_k$ on $Supp(\psi_{k}') \cap Supp(\psi_{k+1}'):$ 
\begin{align*}
\|\partial_x \tilde{f}_{g}^{(N),0}\|^2_{L^2(\Omega)} 
& \lesssim   \sum_{k=1}^{N-1} \left[ \dfrac{1}{N R_{k+1}^0} - \dfrac{1}{NR_k^0} \right]^{2} N  \\
& \lesssim \sum_{k=1}^{N-1} N \left| \int_{c_k^0}^{c_{k+1}^0} \partial_x \bar{f}_g^0(z){\rm d}z \right|^2  \lesssim \|\partial_x \bar{f}_g^0\|^2_{L^2(\Omega)}.
\end{align*}
In these computations,  we use extensively the definitions \eqref{eq_calculid} and also that $|B_k^0|$ and $|\mathcal F_k^0|$ are both of size $O(1/N).$
Similar arguments yield  that:
\[
\|\tilde{\rho}_g^{(N),0}\|^2_{H^1(\Omega)} \lesssim \|\bar{\rho}_g^0\|^2_{H^1(\Omega)}.
\]

Finally,  for arbitrary $\beta \in C^1([0,\infty) \times [0,\infty))$ and $\varphi \in C^{\infty}_c(\Omega),$ we have:
\begin{align*}
\int_{\Omega} \beta(\tilde{\rho}_g^{(N),0},\tilde{f}_g^{(N),0}) \mathds{1}_{\Omega \setminus \tilde{\mathcal F}^{(N),0}} \varphi {\rm d}x
& = \sum_{k=1}^N \int_{B_k^0} \beta(\bar{\rho}_g(c_k^0),\bar{f}_g(c_k^0)) \varphi(x){\rm d}x\\
& = \sum_{k=1}^N 2R_k^0 \beta(\bar{\rho}_g(c_k^0),\bar{f}_g(c_k^0)) \varphi(c_k^0) + O(1/N) \|\partial_x \varphi\|_{L^{\infty}(\Omega)}\\
&= \dfrac{1}{N} \sum_{k=1}^N\dfrac{\beta(\bar{\rho}_g(c_k^0),\bar{f}_g(c_k^0))}{ \bar{f}_g^{0}(c_k^0)} \varphi(c_k^0) + O(1/N) \|\partial_x \varphi\|_{L^{\infty}(\Omega)}
\end{align*}
At this point, we remark that, by construction, we have that 
\[
\dfrac{1}{N} \sum_{k=1}^{N} \delta_{c_k^0} \rightharpoonup \bar{\alpha}_g^0 \bar{f}_g^0 \text{ in $\mathbb P(\Omega)$}.
\]
Since  $t \mapsto \beta(\bar{\rho}_g^0(t),\bar{f}_g^{0}(t))/\bar{f}_g^0(t)$ is continuous on $\bar{\Omega}$ we infer that:
\[
\lim_{N\to \infty} \int_{\Omega} \beta(\tilde{\rho}_g^{(N),0},\tilde{f}_g^{(N),0}) \mathds{1}_{\Omega \setminus \tilde{\mathcal F}^{(N),0}} \varphi 
{\rm d}x= \int_{\Omega} \beta(\bar{\rho}_g,\bar{f}_g^0) \bar{\alpha}_g^0 \varphi {\rm d}x.
\]
This concludes the proof.
\end{proof}

Below, we pick a sequence of initial bubble distribution $(c_k^{(N),0},R_k^{(N),0})_{k=1,\ldots,N}$ and  masses $(m_k^{(N)})_{k=1,\ldots,N}$ given by {\bf Proposition \ref{prop_cstr_id}}.  For any  $N\in \mathbb N,$ assuming the fluid initial data is associated with $\bar{\rho}_f^0,\bar{u}^0,$ we construct initial data for the microscopic system like in \eqref{eq_initbubble}-\eqref{eq_initbubble2}.  We have then that the initial data match the assumptions of {\bf Theorem \ref{thm_existence}} and }
we obtain a solution 
\[
  (\rho_f^{(N)}, u_f^{(N)},
  (c_k^{(N)},R_k^{(N)})_{k\in \{1,\ldots,N\}})
\]
that is defined on a time-span $[0,T]$ which does not depend on $N.$
This creates a sequence of solutions indexed by 
$N$ {whose asymptotic behavior (when $N \to \infty$)} is analyzed in the remaining sections.  

\medskip

Firstly, {\bf Corollary \ref{cor_existence}} entails that we have 
uniform bounds on $[0,T]$ in the form of
\eqref{it_bound_energy}-\eqref{eq_Q6} with a right-hand side $E_0$
independent of $N$, and that \ref{it_bound_R}--\ref{it_H1_bound_tensors}  hold also with a constant $K$ independent of $N$.
In passing, we point out that all the bounds that are derived in Appendix
\ref{app_proof-prop-refpr} and
Appendix \ref{sec_rhoH1} are available since they are obtained under
the sole assumptions that initial data are of the form
\eqref{eq_initbubble}--\eqref{eq_initbubble2} and that the bounds
\ref{it_bound_R}--\ref{it_H1_bound_tensors} hold true. 
Below we denote $\tilde u^{(N)}$ the "mixture" velocity-field meaning that\begin{equation}
  \label{eq_tilde_udfN}
  \tilde u^{(N)} =
  \begin{cases}
    u_f^{(N)},& \text{on } \mathcal F^{(N)},\\
    \dfrac{u_f(x_k^{-,(N)}) + u_f(x_k^{+,(N)})}{2}\\
    \qquad -\dfrac{u_f(x_k^{-,(N)}) -u_f(x_k^{+,(N)})
    }{2R_k^{(N)}}(x-c_k), & \text{on } B_k^{(N)},\; k=1,\dots,N .
  \end{cases}
\end{equation}
Note that the restriction of $\tilde u^{(N)}$ on the bubbles
boils down to
\begin{equation}
  \label{eq_tilde_udfN_PkN}
  \tilde u^{(N)}(\cdot,x) = \dot c_k^{(N)} + \dfrac{\dot
    R_k^{(N)}}{R_k^{(N)}}(x-c_k^{(N)}) \quad \text{on } B_k^{(N)} .
\end{equation}

In what remains of this section, we introduce functions describing the different species and the mixture and we analyse their possible convergences.  Since we use mostly compactness argument below,
all convergence results must be understood "up to the extraction of a subsequence that we do not relabel."  

\subsection{Fluid unknowns}
\label{sec_fluid-unknowns}
In \eqref{eq_macromodel}, the fluid behavior is encoded through its
"volumic fraction" $\bar{\alpha}_f$
and its density $\bar{\rho}_f$.
We recover such quantities from microscopic counterparts.
{We start with the following construction of the volumic fraction:}
\begin{Prop}
    \label{prop_transport_chi}
  Let  $\chi^{(N)}=\mathbb 1_{\mathcal
    F^{(N)}}$. It satisfies
    \begin{equation}
    \label{eq_transport_chi}
    \begin{cases}
      \p_t \chi^{(N)} + \tilde u^{(N)}\p_x \chi^{(N)}=0,& \text{on }
      (0,T)\times \Omega,\\
      \chi^{(N)}(0,.) = \mathbb 1_{\mathcal F^{(N),0}}.
    \end{cases}
  \end{equation}
  Moreover, there exists $\balpha_f\in L^\infty((0,T)\times \Omega)$,
  called the \emph{volumic fraction of the fluid}, such that, up to the
  extraction of a subsequence,
  \begin{equation} \label{eq_alphasup}
    \chi^{(N)} \rightharpoonup\bar \alpha_f \text{ in } L^\infty((0,T)\times
    \Omega)-w^* \quad \text{and} \quad 0 \leq \bar \alpha_f \leq 1-2 d_{\infty}^2/3 \ a.e.
  \end{equation}  
\end{Prop}

\begin{proof}
  Since the fluid domain $\mathcal F^{(N)}$ is transported by the
  velocity field $\tilde u^{(N)}$, \eqref{eq_transport_chi} holds. The
  convergence result is straightforward since the sequence
  $\chi^{(N)}$ is nonnegative and bounded in
  $L^\infty((0,T)\times \Omega)$.  The limit is obviosuly positive. {The only crucial information is the bound from above. For this, we remark that 
  under \ref{it_bound_R}-\ref{it_bound_F},  any sequence of two 
  bubble+fluid intervals has at most length $3 /(d_{\infty}N).$ 
  Hence,  for large $N$,  any segment in $\Omega$ of length $\ell$ contains at least $\ell N d_{\infty}/3 -2$ such sequences in which the volumic proportion of 
  gas-bubbles is at least $2\ell d_{\infty}^2/3+ O(1/N). $ The fluid part of this segment is then asymptotically less than $\ell(1 - 2d_{\infty}^2/3 ).$  
  }
\end{proof}


We point out that a strictly bound from below for $\bar{\alpha}_f$ is also true with similar arguments. We dot not state this bound here since it will not help in the sequel.
For {constructing the macroscopic density, we choose  to extend at first the microscopic density by "filling"} the bubbles in a
sufficiently smooth manner.
To this end, we take advantage
of the fact that $\rho_f^0$ is initially defined (and sufficiently
regular) on the whole $\Omega$.
So, we introduce  $\tilde \rho_f^{(N)}$
as the unique solution to:
\begin{equation}
  \label{eq_trhofN_cauchy}
  \begin{cases}
    \p_t \tilde \rho_f^{(N)} + \tilde u^{(N)} \p_x \tilde \rho_f^{(N)} =
    -\dfrac{\tilde\rho_f^{(N)}}{\mu_f}\left( \tilde \Sigma_f^{(N)}+
      \mathrm{p}_f(\tilde \rho_f^{(N)})\right),& \text{on }(0,T)\times
    \Omega , \\[14pt]
    \tilde \rho_f^{(N)}(0,.)=\bar \rho_f^{0},& \text{on }
    \Omega , 
  \end{cases}
\end{equation}
where $\tSigma_f^{(N)}$ is defined from $\Sigma_f$ by~\eqref{eq_tSigmaf}.
\begin{Prop}
  \label{prop_tilde_rho_fN}
  There  exists a time $T_0<T$, independent of $N$, such that the Cauchy
  problem \eqref{eq_trhofN_cauchy} admits a unique solution
  $\tilde \rho_f^{(N)}\in C([0,T_0]\times \Omega)$. \\
  Moreover, there exists $\brho_f\in L^2((0,{T_0})\times\Omega)$ called the {\em density 
    of the fluid} such that, up to the extraction of a subsequence,
  \[
    \tilde \rho_f^{(N)} \longrightarrow \brho_f \quad \text{in
    }L^2((0,{T_0})\times\Omega) \quad \text{when} \quad
    N\to+\infty.
  \]
\end{Prop}

\begin{proof}
  The well-posedness of the Cauchy problem~\eqref{eq_trhofN_cauchy} is
  guaranteed by the method of characteristics, since $\tu^{(N)}$
  belongs to $L^2((0,T);W^{1,\infty}(\Omega))$. 
  
  The result of convergence is an application of the Aubin--Lions
  lemma. One has to check that:
  \begin{itemize}
  \item $(\tilde \rho_f^{(N)})_N $ bounded in $L^2((0,T);H^1(\Omega))$,
  \item $(\dv_t\tilde \rho_f^{(N)})_N$ bounded in $L^2((0,T);L^2(\Omega))$.
  \end{itemize}
  For the first item,  we apply {\bf
    Proposition \ref{prop_rhoH1}} in Appendix {\ref{sec_rhoH1}}
  which yields that, up to restrict to some time-interval $[0,T_0]
  \subset [0,T]$ we have that $\tilde{\rho}_f^{(N)}$
  satisfies a uniform bound in $L^{\infty}((0,T) ; H^1(\Omega)).$ As
  for the second item,  using directly
  Equation~\eqref{eq_trhofN_cauchy}, a uniform estimate can be
  obtained: 
  \begin{align*}
    &\|\partial_t \tilde{\rho}_f^{(N)}\|_{L^{2}((0,T) \times \Omega)}
      \leq  C_0 \Big[\|\tilde{u}^{(N)}\|_{L^2((0,T) ; H^1(\Omega) )} \|\partial_x
      \tilde{\rho}_f^{(N)}\|_{L^{\infty}((0,T) ; L^2((\Omega))} \\
        &\quad + \dfrac{\|\tilde{\rho}_f^{(N)}\|_{L^{\infty}((0,T) \times \Omega)}}{\mu_f}
      \left(  \|\tilde{\Sigma}_f^{(N)}\|_{L^2((0,T); H^1(\Omega))} + 
          \sqrt{T} \|\tilde{\rho}_f^{(N)}\|^{\gamma}_{L^{\infty}((0,T)
          \times \Omega)}  \right) \Big] ,
  \end{align*}
  where $C_0$ depends only on the parameters of the problem independent of $N$.
  Here again, the right-hand side is uniformly bounded with respect to
  $N$, so that the Aubin--Lions lemma can be applied to deduce the
  existence of the limit $\bar\rho_f$ stated in the proposition.
\end{proof}
 

To illustrate  again that our choice for $\tilde{\rho}_f^{(N)}$ is rigorously adapted,
we mention that, on the fluid domain $\mathcal F^{(N)}$, the definition of the
fluid tensor \eqref{eq_fluid_tensor_sigmaf} gives 
\[
  \dfrac{1}{\mu_f}\left( \tilde \Sigma_f^{(N)}+
    \mathrm{p}_f(\tilde \rho_f^{(N)})\right) = \dv_x u_f^{(N)}.
\]
Moreover, $u_f^{(N)}$ and $\tu^{(N)}$ coincide on
$\mathcal{F}^{(N)}$, and the density $\tilde \rho_f^{(N)}$ is also
solution of
\begin{equation}
  \label{eq_trhofN_cauchy2}
  \begin{cases}
    \p_t \tilde \rho_f^{(N)} + \p_x (\tilde \rho_f^{(N)}  u_f^{(N)})=0
    ,& \text{on }(0,T)\times \mathcal F^{N},\\[8pt]
    \tilde \rho_f^{(N)}(0,.)=\bar \rho_f^{0}.
  \end{cases}
\end{equation}    
As a consequence, the fluid density $\rho_f^{(N)}$ on the fluid domain
$\mathcal F^{(N)}$ is the restriction of the global microscopic
density $\tilde \rho_f^{(N)}$:
\begin{equation}
  \label{eq_trhofN5}
  \rho_f^{(N)} = \tilde\rho_f^{(N)}, \quad \text{on } (0,T)\times
  \mathcal F^{(N)}.
\end{equation}

\subsection{Mixture unknowns}
\label{sec_mixt-unknowns}
We proceed with the construction of unknowns that are involved in
composite equations: a mixture velocity, a mixture density and a
mixture stress tensor.

\medskip

The mixture velocity is deduced from the reconstructed velocity
$\tu^{(N)}$ defined by~\eqref{eq_tilde_udfN}:

\begin{Prop}
  \label{prop_CV_uf}
  There exists $\bu\in L^2((0,T); L^2(\Omega))$ such that, up to the
  extraction of a subsequence,
  \begin{equation*}
    \tu^{(N)} \to \bu \text{ in } L^2((0,T); L^2(\Omega)) \text{ when
    } N\to+\infty.
\end{equation*}
\end{Prop}

\begin{proof}
  This result is an application of the Aubin--Lions
  lemma again. From~\eqref{it_bound_energy} and \eqref{eq_Q6}, the sequence
  $(\tilde u^{(N)})$ is bounded in $L^2((0,T);H^1(\Omega))$. It
  remains to prove {\color{red} a uniform bound for } 
  $(\dv_t\tilde u^{(N)})_N$ in $L^2((0,T);L^2(\Omega))$. By
  \eqref{eq_fluid_NS_1D} and \eqref{eq_tilde_udfN_PkN}, the time
  derivative of the velocity reads:
  \begin{equation*}
    \dv_t\tilde u^{(N)} =
    \begin{cases}
      \displaystyle - u_f^{(N)} \dv_x u_f^{(N)} -  \frac{1}{\rho_f^{(N)}}
      \dv_x \Sigma_f^{(N)} & \text{on }
      \mathcal{F}, \\
      \displaystyle \ddot{c}_k +
      \frac{\ddot{R}_kR_k-(\dot{R}_k)^2}{R_k^2}(x-c_k) -
      \dot{c}_k\frac{\dot{R}_k}{R_k} & \text{on } B_k,
    \end{cases}
  \end{equation*}
  (note that some exponents ${(N)}$ have been removed to lighten the
  notations). Since the velocity $\tilde u^{(N)}$ is continuous
  through the interfaces $c_k\pm R_k$, one has, in
  $\mathcal{D}'((0,T)\times\Omega)$,
  \begin{multline*}
    \dv_t \tilde u^{(N)} = \bigg( - u_f^{(N)} \dv_x u_f^{(N)} -
    \frac{1}{\rho_f^{(N)}}
    \dv_x \Sigma_f^{(N)} \bigg) \mathbb{1}_\mathcal{F} \\
    + \sum_{k=1}^{N} \bigg[ \ddot{c}_k + \bigg(
    \frac{\ddot{R}_k}{R_k}-\frac{(\dot{R}_k)^2}{R_k^2}\bigg) (x-c_k) -
    \dot{c}_k\frac{\dot{R}_k}{R_k} \bigg] \mathbb{1}_{B_k} .
  \end{multline*}
  We now take the $L^2$ norm:
  \begin{align*}
    \| \dv_t\tilde u^{(N)} \|_{L^2(\Omega)}^2
    &\leq  \| \tilde u^{(N)}\|_{L^\infty(\Omega)}^2
      \| \dv_x \tilde u^{(N)}\|_{L^2(\Omega)}^2
      + \frac{1}{|\underline{\rho}_\infty|^2} \|\dv_x \tilde \Sigma_f^{(N)}
      \|_{L^2(\Omega)}^2 \\
    &\quad + 2 \sum_{k=1}^N\bigg[ R_k (\ddot{c}_k)^2 +
      R_k (\ddot{R}_k)^2 + \frac{(\dot{R}_k)^4}{R_k} +
      \frac{(\dot{c}_k\dot{R}_k)^2}{R_k} \bigg] \\
    &\leq C \| \tilde u^{(N)}\|_{H^1(\Omega)}^4  +
    \frac{1}{|\underline{\rho}_\infty|^2} \|\tilde \Sigma_f^{(N)}
      \|_{H^1(\Omega)}^2 \\
    &\quad + 2 \frac{1}{d_\infty M_\infty} \sum_{k=1}^N
      m_k \big( (\ddot{c}_k)^2 + (\ddot{R}_k)^2 \big) + 2 \sum_{k=1}^N
        \frac{1}{R_k} \big( (\dot{R}_k)^4 +
      (\dot{c}_k\dot{R}_k)^2 \big)
  \end{align*}
  by \ref{it_bound_mk0} and \ref{it_bound_R}. 
  Time-integrals of the two first terms on the right-hand side are bounded by \ref{it_bound_H1}
   and \ref{it_H1_bound_tensors} respectively. The third is
  controlled using \ref{it_H1_bound_tensors}. Moreover, by
  \ref{it_bound_mk0}, \ref{it_bound_R}, and then by
  \eqref{it_bound_energy}, the last term can be bounded this way:
  \begin{align*}
    \int_0^T \sum_{k=1}^N \frac{(\dot{R}_k)^2}{R_k}& \big( (\dot{R}_k)^2
    + (\dot{c}_k)^2 \big) \mathrm{d}t \\
    &\leq \frac{1}{d_\infty M_\infty}\int_0^T  \max_{k=1,\dots,N}
      \frac{(\dot{R}_k)^2}{R_k^2}
      \sum_{k=1}^N m_k \big( (\dot{R}_k)^2 + (\dot{c}_k)^2 \big) \mathrm{d}t  \\
    &\leq \frac{2 E_0}{d_\infty M_\infty}\int_0^T  \max_{k=1,\dots,N}
      \frac{(\dot{R}_k)^2}{R_k^2} \mathrm{d}t .
  \end{align*}
  The last right-hand side is finally bounded by using Lemma \ref{lem_RkpsRk}. This concludes the
  proof of the assumptions of the Aubin--Lions lemma, leading to the
  convergence of the sequence $(\tilde u^{(N)})_N$ in $L^2((0,T);L^2(\Omega))$.
\end{proof}

We focus now on the mixture density. For this, we construct the global
density $\rho^{(N)}$:
\begin{equation}
  \label{eq_global_rhoN}
  \rho^{(N)} = \rho_f^{(N)} \mathbb 1_{\mathcal F^{(N)}} +
  \sum_{k=1}^N \rho^{(N)}_k \mathbb 1_{B_k},
\end{equation}
where $ \rho_k^{(N)} = m_k^{(N)}/(2R_k^{(N)})$ is the bubble density
that we reconstruct from the bubble mass and radius.
Notice that the global density $\rho^{(N)}$ belongs to
$L^\infty((0,T)\times \Omega)$, and satisfies a classical mass
conservation law (the proof is left to the reader):
  \begin{equation}
    \label{eq_conv_rhoN}
    \p_t \rho^{(N)} + \p_x (\rho^{(N)}\tu_f^{(N)})=0, \text{ in }
    (0,T)\times \Omega.
  \end{equation}

To conclude,  we 
address 
the asymptotic behavior of extended stresses. 
This is the content of the following proposition:
\begin{Prop}
  \label{prop_CV_Sigma}
  There exist $\tilde \Sigma_f$  and $\tilde \Sigma_g$ in
  $L^2((0,T);H^1(\Omega))$ such that, up to the extraction of a
  subsequence, 
  \begin{equation*}
    \begin{aligned}
      \tilde \Sigma_f^{(N)} &\rightharpoonup \bar \Sigma_f \\
      \tilde \Sigma_g^{(N)} &\rightharpoonup \bar \Sigma_g
    \end{aligned}  \quad  \text{ in
      } L^2((0,T); H^1(\Omega)) \text{ when } N\to+\infty.
  \end{equation*}
\end{Prop}

\begin{proof}
    The estimate \ref{it_H1_bound_tensors} ensures that the sequences $\tilde
  \Sigma_f^{(N)}$ and $\tilde \Sigma_g^{(N)}$ are both bounded in the space $L^2((0,T);
      H^1(\Omega))$. Hence they are relatively compact in $L^2((0,T);
      H^1(\Omega))$ endowed with the weak topology, and the result
      follows. 
\end{proof}

\subsection{Bubble unknowns}
\label{sec_bubb-unknowns}
{We mention first that the indicator of the bubble domains reads $1 - \chi^{(N)}.$ 
Similarly to {Proposition \ref{prop_transport_chi}} we obtain that it converges weakly to some $\bar{\alpha}_g$ satisfying also $0 \leq \bar{\alpha}_g \leq 1$ a.e..  Since $1  =\bar{\alpha}_g + \bar{\alpha}_f,$ {Proposition \ref{prop_transport_chi}} entails further that $\bar{\alpha}_g \geq 2 d_{\infty}^2/3.$

For our analysis, we need a sufficiently strong (pointwise) convergence of bubble density $\rho_g^{(N)}$ and covolume $f_g^{(N)}$ as defined in \eqref{eq_def_dens&cv}.  
Yet,  these quantities are defined only partially on subsets depending on $N.$
To overcome this difficulty,  we note that both quantities satisfy the same continuity equation:
\begin{equation} \label{eq_transportgaz}
\left\{
\begin{aligned}
&  \partial_t \rho_g^{(N)} + \partial_x (\rho_g^{(N)} \tilde{u}^{(N)}) = 0, \\
&  \dv_t f_g^{(N)} + \dv_x (f_g^{(N)} \tilde u^{(N)}) = 0 , 
\end{aligned}
\right.
  \quad \text{ in $\mathcal D'((0,T) \times \Omega)$.}
\end{equation}
We used here in particular that $m_k^{(N)}$  is time-independent and that the bubbles
follow the flow associated with the extended velocity.  We propose then to reproduce the same method we used in the case of fluid unknowns (see Proposition \ref{prop_tilde_rho_fN}). We remark that, on $B_k,$ there holds:
\begin{equation}
  \label{eq_dxu}
  \partial_x \tilde{u}^{(N)} = \dfrac{1}{\mu_g}\Sigma_{k}^{(N)} +  \dfrac{{\kappa_k}}{R_k^{(N)}}. 
\end{equation}
We recall that, on the right-hand side, the first term is the restriction to $B_k$ of the 
extended stress tensor $\tSigma_{g}^{(N)}.$ As for the last term, we wish to extract the
contribution of the pressure and the contribution of the surface
tension for modelling reason (even though keeping the current form would not change the
remark in progress). So we rewrite:
\[
  \dfrac{\kappa_k}{R_k^{(N)}} = {\rm p}_g(\rho_k^{(N)}) +
  \dfrac{\bar{\gamma}_S}{2 N R_k^{(N)}}.
\]
Here, the second term could be related artificially to a density, but it is
usually related to a "covolume" and is treated independently. Actually, this is the reason
motivating the introduction of the unknown $f_g^{(N)}.$ We use now this novel writing of the term $\partial_{x} \tilde{u}^{(N)}$ to see that $(\rho_g^{(N)},f_{g}^{(N)})$ is the restriction of a pair $(\tilde{\rho}_g^{(N)},\tilde{f}_g^{(N)})$ solution to:
\begin{equation} \label{eq_transportgasv2}
     \p_t
      \begin{pmatrix}
        \tilde{\rho}_g^{(N)} \\[2pt]  \tilde{f}_g^{(N)}
      \end{pmatrix}
      + \tilde u^{(N)} \p_x
            \begin{pmatrix}
        \tilde{\rho}_g^{(N)} \\[2pt] \tilde{f}_g^{(N)}
      \end{pmatrix} =
      -\dfrac{1}{\mu_g}      \begin{pmatrix}
       \tilde{\rho}_g^{(N)} \\[2pt]  \tilde{f}_g^{(N)}
      \end{pmatrix}
      \left( \tilde \Sigma_g^{(N)}+
        \mathrm{p}_g(\tilde{\rho}_g^{(N)}) + \bar\gamma_s \tilde{f}_g^{(N)} \right), \text{on }(0,T)\times
      \Omega .\\
\end{equation}
We can then use the stability properties of this latter equation to yield the following proposition:

\begin{Prop} 
  \label{prop_tilde_rhogN_fgN}
  There exists a time $T_0<T$ (independent of $N$) and sequences $(\tilde{\rho}_g^{(N)}, \tilde{f}_g^{(N)}) \in C([0,T_0];H^1(\Omega))$ satisfying the following properties:
  \begin{itemize}
  \item there holds $\rho_g^{(N)} = \tilde{\rho}_g^{(N)}$ and $f_g^{(N)} = \tilde{f}_g^{(N)}$
   on $B_k$ for all $k=1,\ldots,N,$
  \item there exists $(\brho_g,\bar f_g)\in L^2((0,{T_0})\times\Omega)^2$ such
  that, up to the extraction of a subsequence,
  \[
    (\tilde \rho_g^{(N)},\tilde f_g^{(N)}) \longrightarrow
    (\brho_g,\bar f_g) \quad \text{in
    }L^2((0,{T_0})\times\Omega)^2 \quad \text{when} \quad
    N\to+\infty.
  \]
  \end{itemize}
\end{Prop}
%

\begin{proof}
We recall that the initial bubble distribution $(c_k^{(N)},R_k^{(N)})_{k=1,\ldots,N}$
is obtained by applying Proposition \ref{prop_cstr_id} so that they are associated with 
a sequence of initial density/covolume $\tilde{\rho}_g^{(N),0},\tilde{f}_g^{(N),0}$ 
which extend initially $\rho_g^{(N)}$ and $f_g^{(N)}$ and that converge weakly in $H^1(\Omega).$  Hence, we complement \eqref{eq_transportgasv2} with initial condition 
\begin{align}
\tilde{\rho}_g^{(N)}(0,\cdot) = \rho_g^{(N),0} \qquad \tilde{f}_g^{(N)}(0,\cdot) = f_g^{(N),0}  \quad \text{ on $\Omega.$}
 \qquad 
\end{align}
  The result is then proved following exactly the same steps as in the proof
  of Proposition~\ref{prop_tilde_rho_fN}, since
  \ref{it_H1_bound_tensors} involves similar controls on
  $\tilde\Sigma_f^{(N)}$ and $\tilde\Sigma_g^{(N)}$.
\end{proof}
}

\subsection{Two technical lemmas}
\label{sec_technical-lemmas}

We {close this section by providing} two crucial results which allow to pass to
the limit in some nonlinear terms.  {The procedure we apply here is similar to the construction in \cite{BrBuLa}.}

\medskip

Let $b\in C^1([0,1]\times \mathbb R^+\times \mathbb R^+)$ and consider
the sequence
\begin{equation} \label{eq_bN0}
  b^{(N)}(t,x) = b(\chi^{(N)}(t,x),
  \rho^{(N)}(t,x), {f}_g^{(N)}(t,x)), \quad \forall (t,x) \in (0,T)\times \Omega,
\end{equation}
where $\rho^{(N)}$ is defined by \eqref{eq_global_rhoN} and $f_g^{(N)}$ by \eqref{eq_def_dens&cv}.

\begin{Prop} \label{prop_bbar}
There exists $\bar b\in L^\infty((0,T)\times \Omega)$ such
  that, up to a subsequence,
  \begin{equation*}
    b^{(N)}\rightharpoonup \bar b, \quad \text{in } L^\infty((0,T)\times
    \Omega)-w^\star  \text{ when } N\to+\infty. 
  \end{equation*}
  This limit verifies the following identity, for almost every
  $(t,x)\in(0,T)\times\Omega$,
  \begin{equation}
    \label{eq_barb}
    \bar b  = b(1,\bar \rho_f,0)\bar \alpha_f + b(0,\bar{\rho}_g, \bar{f}_g) \bar{\alpha}_g.
  \end{equation}
\end{Prop}

\begin{proof}
 By definition, we have:
\[
    b^{(N)} = b(1,\tilde\rho_f^{(N)},0) \chi^{(N)} +  b
    \bigg(0,\tilde\rho_g^{(N)},\tilde f_g^{(N)}\bigg)(1-\chi^{(N)})
\]
  The strong convergence of $\tilde\rho_f^{(N)}$ (resp. $\tilde{\rho}_g^{(N)}$ and $\tilde{f}_g^{(N)}$),
  see Proposition \ref{prop_tilde_rho_fN}  (resp. Proposition \ref{prop_tilde_rhogN_fgN}) and the weak
  convergence of $\chi^{(N)}$ (Proposition \ref{prop_transport_chi})
  ensure that the first term converges weakly towards
  $b(1,\bar\rho_f,0)\bar\alpha_f$ and the second one to $b(0,\bar{\rho}_g,\bar{f}_g)\bar{\alpha}_g.$
\end{proof}

In the following result, the term $\tilde \Sigma^{(N)}$
denotes either $\tilde \Sigma_f^{(N)}$ or $\tilde \Sigma_g^{(N)}$.

\begin{Prop}
  \label{prop_compensated}
  Assume that $\tilde \Sigma^{(N)}$ converges weakly in
  $L^2((0,T);H^1(\Omega))$, and denote by $\bSigma$ its limit.
  Then for all
  $b\in C^1([0,1]\times \mathbb R^+ \times \mathbb R^+)$, it holds
  \begin{equation*}
    \tilde \Sigma^{(N)}b^{(N)} \rightharpoonup \bar \Sigma \bar b, \quad
    \text{in } \mathcal D'((0,T)\times \Omega) \text{ when } N\to+\infty.
  \end{equation*}
\end{Prop}

\begin{proof}
This result is a variant of so-called "compensated compactness" lemma.
We {can reproduce here the proof of \cite[Lemma 10]{BrHi2} up to adapt 
the definition  of the operator $\partial_x^{-1}$ on mean free functions.}

\end{proof}

\section{Derivation of a macroscopic model}
\label{sec_deriv-macr-model}

Thanks to the results of the previous section, we are now in position
to address the limit $N\to+\infty$ for the microscopic model
\eqref{eq_fluid_mass}--\eqref{eq_droplet_tensor_sigmak}.
Based on the previous definitions of macroscopic unknowns, we derive
successively the various equations of \eqref{eq_macromodel}.
This is the content of the following theorem.

\begin{Thm}
  \label{thm_cvgce}
  Let $\bar{\rho}_f,\balpha_f,\bar{\alpha}_{g}, \bar{\rho}_g, \bar{u}$ be
  as constructed in the previous section. Then,
  we have that $(\balpha_f,\bar{\rho}_f, \balpha_g, \bar{\rho}_g,\bar{f}_g,\bar{u})$
  is a solution to
  \eqref{eq_macromodel}-\eqref{eq_comp}-\eqref{eq_bSigma}-\eqref{eq_bT}
  on $(0,T)$ with initial condition on $\Omega$:
  \begin{align*}
  \balpha_f(0,\cdot) = \bar{\alpha}_f^0 && \balpha_g(0,\cdot) = \bar{\alpha}_f^0\\
  \brho_f(0,\cdot) = \bar{\rho}_f^0 && \balpha_g \brho_g(0,\cdot) = \bar{\alpha}_g^0 \bar{\rho}_g^0 \\
  \bar{u}(0,\cdot) = \bar{u}^0 && \bar{\alpha}_g\bar{f}_g(0,\cdot) = \bar{\alpha}_g^0\bar{f}_g^0
  \end{align*}
  
\end{Thm}

What remains of this section is devoted to the proof of this theorem.
Our first result provides the limit equation for the
limit $\bar b$ associated with an abstract choice of $b$.

\begin{Prop}
  \label{prop_dtbbar}
  Let $b\in C^1([0,1]\times \mathbb R^+\times \mathbb R^+)$ and define
  \begin{equation*}
    \begin{aligned}
      b_{1,f} (z,\xi,\nu) &= (\p_2 b(z,\xi,\nu)\xi+\p_3 b(z,\xi,\nu)\nu-b(z,\xi,\nu))z,\\
      b_{1,g} (z,\xi,\nu) &= (\p_2 b(z,\xi,\nu)\xi+\p_3 b(z,\xi,\nu)\nu-b(z,\xi,\nu))(1-z),\\
      b_{2,f} (z,\xi,\nu)  &=   (\p_2 b(z,\xi,\nu)\xi+\p_3 b(z,\xi,\nu)\nu-b(z,\xi,\nu))z \mathrm{p}_f(\xi),\\
      b_{2,g} (z,\xi,\nu)  &=   (\p_2 b(z,\xi,\nu)\xi+\p_3
      b(z,\xi,\nu)\nu-b(z,\xi,\nu))(1-z) (\mathrm{p}_g(\xi)+\bar
      \gamma_s \nu).
    \end{aligned}
  \end{equation*}
  Then, the limit $\bar b$ defined in Proposition~\ref{prop_bbar}
  satisfies the equation
  \begin{equation}
  \label{eq_dtbbar}
  \left\{
  \begin{aligned}
 & \p_t \bar b + \p_x (\bar u \bar b) +   \dfrac{1}{\mu_f} \left(
      \bar b_{1,f} \bar \Sigma_f  + \bar b_{2,f}  \right)
    + \dfrac{1}{\mu_g}\left(
      \bar b_{1,g}\bar\Sigma_g  +  \bar b_{2,g} \right)=0\\
 &      \bar{b}(0,\cdot) = \balpha_f^0 b(1,\brho_f^0,0) + \balpha_{g}^0 b(0,\brho_g^0,\bar{f}_g^0)
           \end{aligned}
\right.      
  \end{equation}
\end{Prop}

\begin{proof}
 Let us compute for arbitrary $N \in \mathbb N$
  \begin{align*}
    \dv_t b(\chi^{(N)}, \rho^{(N)}, f_g^{(N)}) &= \dv_1 b^{(N)} \dv_t
                                                 \chi^{(N)} + \dv_2 b^{(N)}
                                                 \dv_t \rho^{(N)} +
                                                 \dv_3 b^{(N)} \dv_t
                                                 {f}_g^{(N)} \\
                                               &= - \dv_1 b^{(N)} \tilde u^{(N)} \dv_x \chi^{(N)}
                                                     - \dv_2 b^{(N)}  \dv_x (\rho^{(N)}
                                                 \tilde u^{(N)})
                                                     - \dv_3 b^{(N)}  \dv_x ({f}_g^{(N)}
                                                 \tilde u^{(N)})
  \end{align*}
  by \eqref{eq_transport_chi}, \eqref{eq_conv_rhoN} and
  \eqref{eq_transportgaz}. As a result, we obtain:
 \begin{multline} \label{eq_dtb}
    \dv_t b(\chi^{(N)}, \rho^{(N)}, {f}_g^{(N)}) + \dv_x
    (b(\chi^{(N)}, \rho^{(N)}, {f}_g^{(N)}) \tilde u^{(N)}) \\
    + \big(\dv_2b(\chi^{(N)}, \rho^{(N)}, f_g^{(N)}) \rho^{(N)} + \dv_3b(\chi^{(N)}, \rho^{(N)}, {f}_g^{(N)})
    f_g^{(N)}\\  - b(\chi^{(N)}, \rho^{(N)}, {f}_g^{(N)}) \big) \dv_x
    \tilde u^{(N)} = 0,
  \end{multline}
  in $\mathcal D'((0,T)\times \Omega)$.  In this equation, due to
  the weak convergence of $b^{(N)}$ and the strong convergence of
  $\tu_f^{(N)}$, respectively stated in Propositions \ref{prop_bbar} and
  \ref{prop_CV_uf}, it holds that:
  \begin{equation*}
    \begin{cases}
      b^{(N)} \rightharpoonup \bar{b},  \\
      u_f^{(N)} b(\chi^{(N)},\rho^{(N)},{f}_g^{(N)}) \rightharpoonup  \bar{u}\, \bar{b} ,
    \end{cases}
    \qquad 
    \text{in }\mathcal D'((0,T) \times \Omega).
  \end{equation*}
  Then,  we rewrite:
  \begin{align*}
    \p_2 b(\chi^{(N)},
    \rho^{(N)}, f_g^{(N)}) \rho^{(N)}&+ \p_3b(\chi^{(N)}, \rho^{(N)},
    f_g^{(N)}) f_g^{(N)}-b^{(N)}\p_xu_f^{(N)} \\
    &= \dfrac{1}{\mu_f} \left(
      b_{1,f}^{(N)} \tilde \Sigma_f^{(N)} +  b_{2,f}^{(N)}  \right)
    + \dfrac{1}{\mu_g}\left(
      b_{1,g}^{(N)} \tilde \Sigma_g^{(N)} +  b_{2,g}^{(N)}  \right).
  \end{align*}
  The weak convergence stated in Proposition \ref{prop_compensated}
  allows to pass to the limit the right-hand side, leading to
  \begin{align*}
    \p_2 b(\chi^{(N)},
    \rho^{(N)}, f_g^{(N)})\rho^{(N)} &+ \p_3b(\chi^{(N)}, \rho^{(N)},
    f_g^{(N)}) f_g^{(N)}-b^{(N)}\p_xu_f^{(N)}\\
    &\rightharpoonup
    \dfrac{1}{\mu_f} \left(
      \bar b_{1,f} \bar \Sigma_f  + \bar b_{2,f}  \right)
    + \dfrac{1}{\mu_g}\left(
      \bar b_{1,g}\bar\Sigma_g  +  \bar b_{2,g} \right),
  \end{align*}
  where the terms $\bar b_{1,f}$, $\bar b_{1,g}$, $\bar b_{2,f}$, and
  $\bar b_{2,g}$ are defined as in Proposition~\ref{prop_bbar}.  This
  provides Equation \eqref{eq_dtbbar} for $\bar b$.
  
Finally, we have initially
\[
b^{(N)}(0,\cdot) = \chi^{(N),0} b(1,\rho_f^0,0) + (1 - \chi^{(N),0}) b(0,\tilde{\rho}_g^{(N),0},\tilde{f}_g^{(N),0})
\]  
and we are in position to apply {Proposition \ref{prop_cstr_id}} to pass to the limit in this identity when $N \to \infty.$
\end{proof}

Let us recall that the link between the limit $\bar b$ and the
function $b$ is provided in Proposition \ref{prop_bbar}. According to
the choice of $b$, different relevant macroscopic equations can be
obtained.

\begin{Cor}
  \label{cor_vol_mass}
 The volumic fractions satisfy the following equations
  \begin{equation}
    \label{eq_baralpha}
    \left\{
\begin{aligned}  
   &  \p_t \bar \alpha_f + \p_x (\bar \alpha_f \bar u)= \dfrac{\bar
      \alpha_f}{\mu_f}\left( \bar \Sigma_f + \mathrm{p}_f(\bar \rho_f)\right), 
      \qquad \bar{\alpha}_f(0,\cdot) = \bar{\alpha}_f^0 \\
    &   \balpha_f + \balpha_g = 1 
      \end{aligned}
      \right.
  \end{equation}
The covolume unkwnown $\bar{f}_g$ satisfies the conservation equation:
 \begin{equation*}
    \p_t (\bar \alpha_g \bar f_g) + \p_x(\bar \alpha_g \bar f_g \bar u)=0, 
    \qquad \balpha_g(0,\cdot) \bar f_g(0,\cdot) = \bar{\alpha}_g^0 \bar{f}_g^0.
  \end{equation*}
  The mass conservation laws of both phases read
  \begin{align}
    \label{eq_massconservation}
    \p_t (\bar \alpha_f\bar \rho_f)+ \p_x (\bar \alpha_f\bar
    \rho_f\bar u)= 0, \qquad  \bar \alpha_f(0,\cdot)\bar \rho_f(0,\cdot) = \bar \alpha_f^0 \bar \rho_f^0 \\
    \p_t (\bar \alpha_g\bar \rho_g)+ \p_x (\bar \alpha_g\bar \rho_g\bar u)= 0,
    \qquad \bar \alpha_g(0,\cdot)\bar \rho_g(0,\cdot)= \bar \alpha_g^0\bar \rho_g^0.
  \end{align}
\end{Cor}

\begin{proof}
  By Proposition~\ref{prop_dtbbar}, it suffices to compute the
  different terms of Equation~\eqref{eq_dtbbar}. In the first case, we
  consider $b(z,\xi,\nu) = z$. It yields $\bar b=\balpha_f$ and
  \begin{equation*}
    \begin{aligned}
      b_{1,f} (1,r) &= -1, &
      b_{1,g} (1,r) &= 0, &
      b_{2,f} (1,r)  &=   - \mathrm{p}_f(r), &
      b_{2,g} (1,r)  &=  0,\\
      b_{1,f} (0,r) &= 0, &
      b_{1,g} (0,r) &= 0, &
      b_{2,f} (0,r)  &=   0, &
      b_{2,g} (0,r)  &=  0.
    \end{aligned}
  \end{equation*}
  Computing the associated limits, one recovers the first equation of~\eqref{eq_baralpha}.
  The second equation is true by construction.
  The equation on $\bar f_g$ is obtained in the same way, taking
  $b(z,\xi,\nu)=\nu$.
  Finally the phasic mass conservation laws are derived using
  $b(z,\xi,\nu)=z\xi$ and $b(z,\xi,\nu)=(1-z)\xi$ respectively.
\end{proof}

\subsection{Momentum equation and closure laws}
\label{sec_momentum-closure}
We proceed with the derivation of the momentum equation. 

\begin{Prop}
  Let $\bar \rho=\bar \alpha_f \bar \rho_f+\bar \alpha_g \bar \rho_g$
  be the mixture density. The mixture momentum equation reads
  \begin{equation} \label{eq_moment_asym}
    \partial_t (\bar{\rho} \bar{u}) + \partial_x (\bar{\rho}
    \bar{u}^2) = \partial_x (\bar{\alpha}_f \bar{\Sigma}_f +
    \bar{\alpha}_g \bar{\Sigma}_g),
  \end{equation}
  with
  \begin{equation}
    \partial_x \bar{u} = \dfrac{\bar{\alpha}_f}{\mu_f} \left[
      \bar{\Sigma}_f + \mathrm{p}_f(\bar{\rho}_f)\right]
    + \dfrac{\bar \alpha_g}{\mu_g}\left[\bar{\Sigma}_g +
      \mathrm{p}_g(\bar\rho_g)+\bar{\gamma}_s  \bar{f}_g\right],
  \end{equation}
  and
  \begin{equation}
    \label{eq_id_sigma}
    \bar{\Sigma}_f = \bar{\Sigma}_g.
  \end{equation}
\end{Prop}

\begin{proof}
  Let us consider the momentum equation in the fluid domain and
  multiply it by a test function $w\in C_c^\infty((0,T)\times
  \Omega)$. It yields
  \begin{equation*}
    \int_0^T \int_{\mathcal F^{(N)}(t)} (\p_t (\rho_f^{(N)}u_f^{(N)})
    + \p_x (\rho_f^{(N)}|u_f^{(N)}|^2))w \mathrm{d}x \mathrm{d}t =
    \int_0^T \int_{\mathcal F^{(N)}(t)}  \p_x \Sigma^{(N)}_f w \mathrm{d}x \mathrm{d}t.
  \end{equation*}
  Since the fluid domain $\mathcal F^{(N)}(t)$ is transported 
  with the velocity $u_f^{(N)}$, an integration by part in time of the
  left-hand side gives
  \begin{align*}
     \int_0^T \int_{\mathcal F^{(N)}(t)} (\p_t (\rho_f^{(N)}u_f^{(N)})
    &+ \p_x (\rho_f^{(N)}|u_f^{(N)}|^2))w \mathrm{d}x \mathrm{d}t \\
    &= -
    \int_0^T \int_{\mathcal F^{(N)}(t)}
    \rho_f^{(N)}u_f^{(N)}
    (\p_t w+ u_f^{(N)}\p_x w) \mathrm{d}x \mathrm{d}t.
  \end{align*}
The right-hand side is handled by an integration by part in
space. Reorganising the boundary terms yields (we omit time dependencies for simplicity):
\begin{align*}
   \int_0^T \int_{\mathcal F^{(N)}(t)}  \p_x \Sigma_f^{(N)} w \mathrm{d}x
   \mathrm{d}t
   =& - \int_0^T \sum_{k=1}^N (\Sigma_f^{(N)}(x_k^+) w(x_k^+) -
  \Sigma_f^{(N)}(x_k^-)w(x_k^-)) \mathrm{d}t \\
  &- \int_0^T \int_{\mathcal
     F^{(N)}}\Sigma_f^{(N)}\p_x w \mathrm{d}x
   \mathrm{d}t.
\end{align*}
We now focus on the boundary terms. For $k=1,\ldots,N$, one has
\begin{align*}
  \Sigma_f^{(N)}(x_k^+) w(x_k^+) -
  \Sigma_f^{(N)}(x_k^-)w(x_k^-) &= (\Sigma_f^{(N)}(x_k^+)
                                  -\Sigma_f^{(N)}(x_k^-)) w(c_k) \\
  &+ (\Sigma_f^{(N)}(x_k^+)
    +\Sigma_f^{(N)}(x_k^-)) R_k \p_x w(c_k)\\
  &+
  O(\|\Sigma_f^{(N)}\|_{L^\infty(\Omega)}R_k^2 \|w\|_{C^2}).
  \end{align*}
From the bubbles equations \eqref{eq_droplet_newton_1D_C} and
\eqref{eq_droplet_newton_1D_R}, one deduces
\begin{align*}
  \Sigma_f^{(N)}(x_k^+) w(x_k^+) -
  \Sigma_f^{(N)}(x_k^-)w(x_k^-) &= m_k \ddot c_k w(t,c_k) + \left(
  \frac{m_k}{3}\ddot R_k+ 2 \Sigma_k\right)R_k \p_x w(t,c_k)\\
  &+ O(\|\Sigma_f^{(N)}\|_{L^\infty(\Omega)}R_k^2 \|w\|_{C^2}).
\end{align*}
The term involving the stress tensor can be rewritten as follows
\begin{equation*}
  2 \Sigma_kR_k \p_x w(t,c_k)=\int_{B_k}\Sigma_g^{(N)}\p_x w \mathrm{d}x+ O(\|\Sigma_g^{(N)}\|_{L^\infty(\Omega)}R_k^2 \|w\|_{C^2}).
\end{equation*}
Therefore, one has
\begin{align*}
  &- \int_0^T \sum_{k=1}^N (\Sigma_f^{(N)}(x_k^+) w(x_k^+) -
    \Sigma_f^{(N)}(x_k^-)w(x_k^-)) \mathrm{d}t \\
  &= - \int_0^T \sum_{k=1}^N \bigg( m_k \ddot c_k w(t,c_k) +
    \frac{m_k}{3}\ddot R_kR_k \p_x w(t,c_k) +
    \int_{B_k}\Sigma_g^{(N)}\p_x w \mathrm{d}x\\
  &\quad+O((\|\Sigma_f^{(N)}\|_{L^\infty(\Omega)}+\|\Sigma_g^{(N)}\|_{L^\infty(\Omega)})R_k^2 \|w\|_{C^2})\bigg)\mathrm{d}t.
\end{align*}
An integration by part in time gives
\begin{align*}
  &- \int_0^T \sum_{k=1}^N (\Sigma_f^{(N)}(x_k^+) w(x_k^+) -
    \Sigma_f^{(N)}(x_k^-)w(x_k^-)) \mathrm{d}t \\
  &= \int_0^T \sum_{k=1}^N m_k \bigg( |\dot c_k|^2 \p_x w(t,c_k) + \dfrac
  1 3 |\dot R_k|^2 \p_x w(t,c_k) + \frac 1 3 \dot R_k R_k \dot c_k
  \p_{xx}w(t,c_k)\bigg)\mathrm{d}t\\
  &\quad + \int_0^T \sum_{k=1}^N m_k \bigg( \dot c_k \p_t w(t,c_k) +
    \frac 1 3 \dot R_k R_k \p_{xt} w(t,c_k)\bigg) \mathrm{d}t\\
  & \quad - \int_0^T \int_{\Omega\setminus \mathcal F^{(N)}}
    \Sigma_g^{(N)}\p_x w \mathrm{d}x\mathrm{d}t \\
  & \quad +
    O\bigg((\|\Sigma_f^{(N)}\|_{L^2((0,T),H^1(\Omega))}+\|\Sigma_g^{(N)}\|_{L^2((0,T),H^1(\Omega))})\sqrt{T}
    \|w\|_{C^2}\max_{[0,T]}\sum_{k=1}^N R_k^2\bigg)\\
  &= \int_0^T \sum_{k=1}^N m_k \bigg( |\dot c_k|^2 \p_x w(t,c_k) + \dfrac
  1 3 |\dot R_k|^2 \p_x w(t,c_k) +\dot c_k \p_t w(t,c_k) \bigg)\mathrm{d}t
  \\
  &\quad- \int_0^T \int_{\Omega\setminus \mathcal F^{(N)}}
    \Sigma_g^{(N)}\p_x w \mathrm{d}x\mathrm{d}t \\
  & \quad +O\bigg((\|\Sigma_f^{(N)}\|_{L^2((0,T),H^1(\Omega))}+\|\Sigma_g^{(N)}\|_{L^2((0,T),H^1(\Omega))})\sqrt{T}
      \|w\|_{C^2}{(d_\infty N)^{-1}}\\
  & \qquad\qquad+{\color{red} (M_\infty N)^{-\frac 12}}\|w\|_{C^2}T\sqrt{E_0}\bigg).
\end{align*}
{\color{red} where we applied \ref{it_bound_mk0} and \eqref{it_bound_energy}
 with \ref{it_bound_R} to yield the last term in the last inequality.}
 On the bubble $B_k$, it holds
\begin{align*}
 & \int_{B_k} \rho_k \tilde u^{(N)} (\p_t w + \tilde u^{(N)}\p_x
  w)\mathrm{d}x \\
  &=\int_{B_k} \frac{m_k}{2R_k} \tilde u^{(N)} (\p_t w + \tilde u^{(N)}\p_x
    w)\mathrm{d}x \\
  &= m_k \dot c_k \p_t w(c_k) + m_k (|\dot c_k|^2 + \dfrac 1 3 |\dot
    R_k|^2)\p_x w(c_k) \\
  &\quad + O\bigg(m_k\|w\|_{C^2}(1+|\dot c_k|+|\dot R_k| )(|\dot{c}_k| + |\dot{R}_k|)| R_k|\bigg).
\end{align*}
Gathering the fluid and gas expressions yields
\begin{align*}
  & -\int_0^T \int_{\Omega}
    \rho^{(N)}\tu^{(N)}
  (\p_t w+ \tu^{(N)}\p_x w) \mathrm{d}x \mathrm{d}t\\
  &=-\int_0^T \int_\Omega
    (\chi^{(N)}\tilde\Sigma_f^{(N)}+(1-\chi^{(N)})\tilde\Sigma_g^{(N)})\p_x
    w \mathrm{d}x \mathrm{d}t + O(N^{-1/2}).
\end{align*}
Using the strong convergence of $\tu^{(N)}$ and the weak convergence
of $\rho^{(N)}$, obtained by Proposition \ref{prop_bbar} with
$b(z,\xi,\nu)=\xi$, the left-hand side tends to
\begin{equation*}
  -\int_0^T \int_\Omega \bar \rho\bar u(\p_tw + \bar u \p_x
  w)\mathrm{d}x \mathrm{d}t.
\end{equation*}
The limit of the right-hand side is deduced from Proposition
\ref{prop_compensated}. One ends up with the desired momentum equation
\eqref{eq_moment_asym}.

It remains to close the system by determining relations between the tensors $\bar
\Sigma_f$ and $\bar \Sigma_g$ and the other quantities. To do so, we
prove that $\bar\Sigma_f$ and $\bar \Sigma_g$ are solutions of a $2\times
2$ system.

First observe that
\begin{equation*}
  \p_x u^{(N)} = \chi^{(N)} \dfrac{\tilde \Sigma_f^{(N)}+
    \mathrm{p}_f(\rho_f^{(N)})}{\mu_f}+
  (1-\chi^{(N)})  \dfrac{\tilde \Sigma_g^{(N)}+
    \mathrm{p}_g(\rho_g^{(N)})+ F_s/2}{\mu_g}.
\end{equation*}
The different results of convergence given in Section
\ref{sec_homog-probl},  especially Proposition
\ref{prop_compensated}, allow to pass to the limit in both sides of
the equation.
In particular, in the right-hand side, the definition of the surface
tension yields
\[
  (1-\chi^{(N)}) \frac{F_s}{2} = \sum_{k=1}^N\dfrac{\bar
  \gamma_s}{2NR_k}\mathbb 1_{B_k} =  \bar{\gamma}_s \tilde{f}_g^{(N)} (1-\chi^{(N)}) \rightharpoonup\bar \gamma_s \bar{\alpha}_g \bar{f}_g. 
\]

Eventually, it holds
\begin{align*}
  \p_x \bar u =  \dfrac{\bar \alpha_f}{\mu_f}\big[ \bar \Sigma_f +
  \mathrm{p}_f(\bar \rho_f)\big] + \dfrac{\bar \alpha_g}{\mu_g}\big[ \bar \Sigma_g+
  \mathrm{p}_g(\bar \rho_g)+ \bar \gamma_s \bar f_g\big].
\end{align*}
The second equation is obtained while studying the difference $\bar
\Sigma_f-\bar \Sigma_g$.
Using the definition \eqref{eq_tSigmaf} of the extended tensor $\tilde
\Sigma_f$ and the Newton laws \eqref{eq_droplet_newton_1D_C} and
\eqref{eq_droplet_newton_1D_R} for the bubbles, it holds
\begin{align*}
  \tilde \Sigma_f^{(N)} &= \dfrac{\Sigma_f(x_k^-)+
                          \Sigma_f(x_k^+)}{2}- \dfrac{\Sigma_f(x_k^-)-
                          \Sigma_f(x_k^+)}{2R_k^{(N)}}(c-c_k^{(N)})\\
                        &=\frac{m_k}{6}\ddot R_k + \Sigma_k +
                          \dfrac{m_k\ddot c_k}{2R_k}(x-c_k).
\end{align*}
Since $\tilde\Sigma_g^{(N)}=\Sigma_k$ on the bubbles domain $B_k$, one
has
\begin{equation*}
  (1-\chi^{(N)})(\tilde \Sigma_f^{(N)}-\tilde \Sigma_g^{(N)})= 
    \sum_{k=1}^N \left(\frac{m_k}{6}\ddot R_k +
    \dfrac{m_k\ddot c_k}{2R_k}(x-c_k)\right)\mathbb 1_{B_k}.
\end{equation*}
Proposition \ref{prop_compensated} applies to the left-hand side:
\begin{equation*}
    (1-\chi^{(N)})(\tilde \Sigma_f^{(N)}-\tilde
    \Sigma_g^{(N)})\rightharpoonup (1-\bar \alpha_f) (\bar \Sigma_f
    -\bar \Sigma_g),
\end{equation*}
in the sense of distributions.
The right-hand side can be proved to tend to zero in $L^2((0,T)\times
\Omega))$ since
\begin{align*}
  \bigg\| \sum_{k=1}^N \left(\frac{m_k}{6}\ddot R_k +
    \dfrac{m_k\ddot c_k}{2R_k}(x-c_k)\right)\mathbb
  1_{B_k}\bigg\|_{L^2(\Omega)}^2
  &\leq \bigg\| \frac 1 2 \sum_{k=1}^N m_k\left(|\ddot R_k |+
    |\ddot c_k|\right)\mathbb
    1_{B_k}\bigg\|_{L^2(\Omega)}^2\\
  & \leq\sum_{k=1}^N \int_{B_k}m_k^2 (|\ddot R_k |^2+|\ddot
    c_k|^2)\mathrm{d}x\\
  &\leq \frac{2}{N^2d_\infty M_\infty}\sum_{k=1}^N m_k (|\ddot R_k |^2+|\ddot
    c_k|^2)\\
  &\leq \frac{2K_\infty}{N^2d_\infty M_\infty},
\end{align*}
thanks to \ref{it_bound_mk0}, \ref{it_bound_R} and
\ref{it_H1_bound_tensors}.
Recalling the second part of \eqref{eq_alphasup}, one recovers \eqref{eq_id_sigma}.
\end{proof}

\section{An alternative description of the bubble dynamics}

 In order to describe the dynamics of the bubbles,  an alternative approach is to introduce the distribution function in position and (scaled) radius
 
\begin{equation}
  \label{eq_StN}
  S_t^{(N)} = \dfrac 1 N \sum_{k=1}^N \delta_{c_k(t), NR_k(t)},
\end{equation}
which is a measure on $\Omega\times (0,\infty)$.  According to
\ref{it_bound_R}, one has
\begin{equation*}
  \text{supp}(S_t^{(N)})\subset \bar \Omega \times [d_\infty,
  1/d_\infty],
  \quad \forall\, t \in (0,T).
\end{equation*}

\begin{Prop}
  \label{prop_stN_equicont}
  For all $\beta\in C(\bar \Omega \times [d_\infty, 1/d_\infty])$, the
  distribution function $ S_t^{(N)}$ satisfies
  \begin{equation}
    \label{eq_dtSN}
    \p_t \langle S_t^{(N)},\beta\rangle - \langle S_t^{(N)},\tilde
    u^{(N)}(x) \p_x\beta \rangle 
    - \dfrac{1}{\mu_g}\langle S_t^{(N)}, \big((\tilde \Sigma_g^{(N)}(x)+
    \mathrm{p}_g(\tilde\rho_g^{(N)})) r + \bar\gamma_s/2 \big) \p_r\beta \rangle = 0.
  \end{equation}
  Moreover, the sequence of applications $t\mapsto S_t^{(N)}$ is
  compact in
  $C([0,T]; \mathbb P(\bar \Omega \times [d_\infty, 1/d_\infty]))$. As
  a consequence, there exists
  $\bS_g\in C([0,T]; \mathbb P(\bar \Omega \times [d_\infty,
  1/d_\infty]))$ such that, up to the extraction of a subsequence,
\begin{equation*}
  \langle S^{(N)},\beta\rangle \to \langle \bar S_g,\beta\rangle,
  \quad \text{in } C([0,T]),
\end{equation*}
for any $\beta\in C(\bar \Omega \times [d_\infty, 1/d_\infty])$.
\end{Prop}

\begin{proof}
  Let us first prove that, for any
  $\beta\in C^{1}(\bar \Omega \times [d_\infty, 1/d_\infty])$, the
  sequence $\beta^{(N)}:t\mapsto \langle S_t^{(N)},\beta\rangle$ is
  uniformly equicontinuous. The definition of $\beta^{(N)}$ and~\eqref{eq_StN}
  enable to write
  \begin{equation*}
    \beta^{(N)}(t) = \dfrac 1 N \sum_{k=1}^N \beta (c_k(t), NR_k(t)).
  \end{equation*}
  For legibility, we drop the exponent $(N)$ in $c_k$ and $R_k$ 
  here and in what remains of the proof.
  By construction, $c_k$ and $R_k$ belong to $H^2(0,T)$
  and thus are in $C^1([0,T])$. It follows that
  $\beta^{(N)}\in C^1([0,T])$, and
  \begin{equation}
    \label{eq_dtbetaN}
    \begin{aligned}
      \dfrac{\mathrm{d}}{\mathrm{dt}}\beta^{(N)}(t) &= \dfrac 1 N
      \sum_{k=1}^N \left( \dot c_k \p_x \beta(c_k, NR_k) + N\dot R_k
        \p_r \beta(c_k,NR_k)\right)\\
      &= \dfrac 1 N \sum_{k=1}^N \left(  \tilde u_f^{(N)}(c_k) \p_x
        \beta(c_k, NR_k)
        + \dfrac{\dot R_k}{R_k}NR_k\p_r\beta(c_k, NR_k)\right).
    \end{aligned}
  \end{equation}
  Recall that, by Corollary \ref{corollaire_bound_energy},
  $\tilde u_f^{(N)}$ is bounded in $L^2((0,T);H^1(\Omega))$ and then
  in $L^2((0,T);C(\bar \Omega))$. Moreover \ref{it_bound_R} ensures
  that $NR_k$ is bounded by $d_\infty$. From Lemma
  \ref{lem_RkpsRk}, $\dfrac{\mathrm{d}}{\mathrm{dt}}\beta^{(N)}$ is
  bounded in $L^2(0,T)$. Thus $\beta^{(N)}$ is bounded in $H^1(0,T)$
  and then uniformly equicontinuous. The compactness result and the
  existence of $\bS_g$ is then straightforward.

  It remains to check that $S_t^{(N)}$ verifies
  equation~\eqref{eq_StN}. This comes directly from~\eqref{eq_dtbetaN}
  where the term $\dot R_k/R_k$ is replaced
  using~\eqref{eq_droplet_tensor_sigmak}:
  \begin{align*}
    \frac{\dot R_k}{R_k} &= \frac{1}{\mu_g} \bigg( \Sigma_k +
    \mathrm{p}_g(\rho_k) + \frac{\bar\gamma_s}{2}\frac{1}{NR_k} \bigg) \\
                         &= \frac{1}{\mu_g} \bigg( \tilde \Sigma_g(c_k) +
    \mathrm{p}_g(\tilde \rho_g(c_k)) + \frac{\bar\gamma_s}{2}\frac{1}{NR_k} \bigg) .
  \end{align*}
\end{proof}

Actually, the dependence of the measures $\bS_{g,t}$ with respect to
the space variable $x$ can be precised:

\begin{Prop}
  \label{prop_barSt_density}
  For any $\beta\in C^\infty(\mathbb R^+)$, there exists
  $\bar S_\beta\in L^\infty ((0,T);L^\infty(\Omega))$ such that,
  for all $ \Phi\in C_c^\infty((0,T)\times \Omega)$ and all $t\in(0,T)$,
  \begin{equation}
    \label{eq_Sbeta_density}
    \langle \bar S_{g,t} , \Phi(t,\cdot)\otimes\beta\rangle
      = \int_\Omega \bar S_\beta(t,x)\Phi(t,x) \mathrm{d}x. 
  \end{equation}
  In other words, we have:
  \begin{equation*}
    \bar S_\beta(t,\cdot) = \int_{\mathbb R^+} \beta(r)\bar S_{g,t}
    (\cdot,\mathrm{d}r) \in L^\infty ((0,T) \times \Omega).
  \end{equation*}
\end{Prop}

\begin{proof}
  Let $\beta\in C^\infty(\mathbb{R}^+)$ and $\phi\in
  C_c^\infty((0,T)\times\Omega)$. One has for every $t\in(0,T)$
  \begin{align*}
    \langle S_{g,t}^{(N)} , \phi(t,\cdot) \otimes \beta \rangle & = \frac1N
                                                                  \sum_{k=1}^N \beta(NR_k(t)) \phi(t,c_k(t)) \\ 
                                                                & = \sum_{k=1}^N \int_{B_k} \dfrac{1}{2NR_k} \beta(NR_k) \phi(t,x){\rm
                                                                  d}x \\ 
                                                                & \quad - \sum_{k=1}^N \int_{B_k} \dfrac{1}{2NR_k} \beta(NR_k) ( \phi(t,x) - \phi(t,c_k)).
  \end{align*}
  The second term of the right-hand side can be bounded by
  \begin{equation*}
    \left(\max_{k=1,\ldots,N} R_k \right) \|\beta\|_{L^\infty([d_{\infty},1/d_{\infty}])}
    \|\dv_x\phi\|_{L^\infty((0,T);L^\infty(\Omega))},
  \end{equation*}
  and then tends to $0$ when $N\to+\infty$ (see \ref{it_bound_R}). The first term can be
  written as
  \begin{equation*}
    \int_\Omega S_\beta^{(N)}(t,x)\phi(t,x) \mathrm{d}x
  \end{equation*}
  with
  \[
    S_{\beta}^{(N)}(t,x) = \sum_{k=1}^N\dfrac{1}{2NR_k(t)} \beta(NR_k(t))
    \mathbb{1}_{B_k(t)}(x)
  \]
  which provides a bounded sequence in $L^\infty((0,T)\times\Omega)$,
  by \ref{it_bound_R}. Therefore, there exists $\bar
  S_\beta\in L^\infty((0,T)\times\Omega)$ such that, up to the
  extraction of a subsequence,
  \begin{equation*}
    S_{\beta}^{(N)} \rightharpoonup \bar S_\beta, \quad   \text{in
    } L^\infty((0,T)\times \Omega)-w^\star \text{ when } N\to+\infty. 
  \end{equation*}
  Then, letting $N\to\infty$ in the previous equality
  yields~\eqref{eq_Sbeta_density}.
\end{proof}
 

We obtain then the following limiting equation for $\bar{S}_{g,t}:$
\begin{Prop}
  The limit $\bar S_{g,t}$ defined in Proposition
  \ref{prop_stN_equicont} satisfies the equation
  \begin{equation}
  \label{eq_dtSbar}
  \p_t \bar S_{g,t} + \p_x (\bar S_{g,t} \bar u)
  + \dfrac{1}{\mu_g} \p_r((r (\bar \Sigma_g+\mathrm{p}_g(\bar\rho_g))
  + \bar\gamma_s/2) \bar S_{g,t})=0.
\end{equation}
\end{Prop}

\begin{proof}
  To obtain a time-evolution PDE for $\bar S_{g,t}$, we go back to Equation~\eqref{eq_dtSN}
  with a tensorised test function $\beta(x,r)=\beta_x(x)\beta_r(r)$,
  which writes
  \begin{align*}
    \p_t \langle S_t^{(N)},\beta\rangle &= \langle S_t^{(N)},\tilde
    u^{(N)}(x) \beta_x' \otimes \beta_r \rangle \\
    &\quad + \dfrac{1}{\mu_g}\langle S_t^{(N)}, 
    r\tilde \Sigma_g^{(N)}(x) \beta_x \otimes \beta_r' \rangle \\
    &\quad + \dfrac{1}{\mu_g}\langle S_t^{(N)},
      r\mathrm{p}_g(\tilde\rho_g^{(N)}) \beta_x \otimes \beta_r'  \rangle \\
     &\quad + \dfrac{ \bar\gamma_s}{2\mu_g}\langle S_t^{(N)}, 
       \beta_x \otimes \beta_r' \rangle .
  \end{align*}
  The first term of the right-hand side can be dealt using the strong
  convergence of $(\tilde u^{(N)})_N$ in
  $L^2((0,T),L^2(\Omega))$. Since it is bounded in
  $L^2((0,T),H^1(\Omega))$, the sequence $(\tilde u^{(N)})_N$ converges also in
  $L^2((0,T),C(\bar\Omega))$ by interpolation. The weak convergence of
  $(S_t^{(N)})_N$ in
  $C([0,T],\mathbb{P}(\bar\Omega\times\mathbb{R}^+))$ together with
  this strong convergence gives
  \begin{equation*}
    \langle S_t^{(N)} , \tilde u^{(N)} \beta_x' \otimes \beta_r
    \rangle \longrightarrow     \langle \bar S_{g,t} , \bar u \beta_x' \otimes \beta_r
    \rangle .
  \end{equation*}
  Since $\tilde\Sigma_g$ is uniformly bounded in
  $L^2((0,T),H^1(\Omega))\subset L^2((0,T),C^{0,1/2}(\bar\Omega))$,
  the second term writes
  \begin{align*}
    \langle S_t^{(N)}, 
    r\tilde \Sigma_g^{(N)}(x) \beta_x \otimes \beta_r' \rangle &=
                                                                 \frac1N
                                                                 \sum_{k=1}^N 
                                                                 \tilde
                                                                 \Sigma_g^{(N)}
                                                                 (c_k)
                                                                 NR_k
                                                                 \beta_x(c_k)
                                                                 \beta_r'(NR_k) \\
    &= \frac12  \sum_{k=1}^N  \int_{B_k}  \tilde  \Sigma_g^{(N)}(x)  \beta_x(x)
      \beta_r'(NR_k)  \mathrm{d}x + \dfrac{C_{\beta} \|\tilde{\Sigma}_g\|_{H^1(\Omega)}}{\sqrt{N}}\\
    &= \frac12  \int_{\Omega}  \tilde  \Sigma_g^{(N)}(x)  \beta_x(x)
      b^{(N)}(x) \mathrm{d}x +\dfrac{ C_{\beta} \|\tilde{\Sigma}_g\|_{H^1(\Omega)}}{\sqrt{N}}
  \end{align*}
  where $b^{(N)}$ is defined by Equation \eqref{eq_bN0}, with
  \begin{equation*}
    b(1,\cdot,\cdot) = 0, \quad b(0,\cdot,\nu) = \beta_r'(1/(2\nu)) .
  \end{equation*}
  Indeed, this provides
  \begin{equation*}
    b^{(N)} =
    \begin{cases}
      0 &\text{in } \mathcal{F}^{(N)}, \\
      \beta_r'(1/(2f_k)) &\text{in } B_k.
    \end{cases}
  \end{equation*}
    Using successively Propositions \ref{prop_compensated},
    \ref{prop_bbar} and \ref{prop_barSt_density}, we obtain
    \begin{align*}
      \langle S_t^{(N)}, 
    r\tilde \Sigma_g^{(N)}(x) \beta_x \otimes \beta_r' \rangle
      \longrightarrow &\frac12 \int_\Omega \bar \Sigma_g(x) \bar b(x)
                        \beta_x(x) \mathrm{d}x \\
      &=\frac12 \int_\Omega \bar \Sigma_g(x) \bigg[
        b(1,\bar\rho_g,0)\bar\alpha_f \\
      &\quad + \int_{\mathbb{R}^+} (2r)
        b(0,\bar \rho_g,1/(2r)) \bar S_{g,t} (\mathrm{d}r) \bigg]
                        \beta_x(x) \mathrm{d}x \\
      &= \frac12 \int_\Omega \bar \Sigma_g(x) \bigg[ \int_{\mathbb{R}^+} (2r)
        \beta_r'(r) \bar S_{g,t}(\mathrm{d}r) \bigg]
                        \beta_x(x) \mathrm{d}x \\
      &= \langle \bar S_{g,t}, r \bar  \Sigma_g \beta_x \otimes
        \beta_r'\rangle .
    \end{align*}
    For the third term, we proceed similarly, defining
      \begin{equation*}
        b(1,\cdot,\cdot) = 0, \quad b(0,\xi,\nu) = \mathrm{p}_g(\xi)\beta_r'(1/(2\nu)),
      \end{equation*}
      so that
      \begin{equation*}
        \langle S_t^{(N)}, 
        r\mathrm{p}_g(\tilde\rho_g^{(N)}) \beta_x \otimes \beta_r' \rangle
      \longrightarrow 
        \langle \bar S_{g,t}, 
        r\mathrm{p}_g(\bar\rho_g) \beta_x \otimes \beta_r' \rangle .
      \end{equation*}
      The convergence of the last term is nothing else but the
      convergence of $S_t^{(N)}$.
\end{proof}

Observe that $\bar \alpha_g \bar f_g$ and $\bar \alpha_g$ are respectively the
zeroth and first moments of $\bar S_{g}$. Their PDE's, see Corollary
\ref{cor_vol_mass},
can be deduced from the Equation \eqref{eq_dtSbar}.

\newpage
\appendix
\section{Proof of Proposition \ref{prop_estimates}} 
\label{app_proof-prop-refpr}

In the whole section, we consider $T>0$ and
$(\rho_f,u_f,(c_k,R_k)_{k=1,\ldots,N})$ is classical solution to
\eqref{eq_fluid_mass}-\eqref{eq_droplet_tensor_sigmak} on $(0,T)$,
satisfying \ref{it_bound_R}-\ref{it_H1_bound_tensors}.

\medskip

To start with, we recall that {\bf Corollary \ref{corollaire_bound_energy}}
applies. With \ref{it_bound_R}, these estimates
yield:
\begin{align}
\label{it_bound_energy} 
&\displaystyle \int_{\mathcal F} \left( \rho_f
    \dfrac{|u_f|^2}{2}+q(\rho_f)\right)\mathrm{d}x
  + \dfrac{1}{2}\sum_{k=1}^N m_k \big(|\dot c_k|^2 + \dfrac 1 3
  |\dot R_k|^2\big) \\ \notag
  & \phantom{156890876543213456788654321145643}- \sum_{k=1}^N \kappa_k \ln(d_\infty
  N R_k)  \leq E_0,\\
    \label{eq_Q6}
 &   \int_0^T\Bigg[\bigg(\int_{\mathcal F} \mu_f |\p_x u_f|^2 \mathrm{d}x+
      \mu_g \sum_{k=1}^N \dfrac{|\dot R_k|^2}{R_k}\bigg)\Bigg]\mathrm{d}t\leq E_0 ,
\end{align}
with a constant $E_0$ depending only on the list of parameters \eqref{eq_listparametre}.

 \subsection*{Strict version ($\mathscr Q_1$) of \ref{it_bound_R}}
 Since $|a|-|b|\leq |a-b|$ and
 $(\alpha^2+\beta^2+\gamma^2)^{1/2} \leq \alpha
 +(\beta^2+\gamma^2)^{1/2}$ as soon as the $\alpha,\beta$ and $\gamma$
 are nonnegative, it follows from Corollary \ref{cor_est_sigmag}, \ref{it_bound_mk0}  
 and the bounds \ref{it_bound_R} on $R_k,$ \ref{it_bound_F} on $|\mathcal F_k|$ 
 that
\[
  \left|\dfrac{\dot R_k}{R_k} \right|\leq \dfrac{1}{\mu_g}\dfrac{1}{M_{\infty}NR_k} +
  \dfrac{C_1}{\mu_g} \|\tilde \Sigma_g\|_{H^1(\Omega)} +
  \dfrac{C_1}{\mu_g} \left( \dfrac{1}{\min_k |\mathcal F_k|}\sum_{k=1}^N (m_k)^2
    (|\ddot R_k|^2 + |\ddot c_k|^2)\right)^{1/2},
\]
and then:
\begin{equation}
  \label{eq_bound_R_2}
  \left|\dfrac{\dot R_k}{R_k} \right|\leq \dfrac{1}{\mu_g}\dfrac{1}{M_\infty d_\infty} +
  \dfrac{C_1}{\mu_g} \|\tilde \Sigma_g\|_{H^1(\Omega)} +
  \dfrac{C_1}{\mu_g} \left( \dfrac{1}{M_\infty d_\infty}\sum_{k=1}^N
    m_k (|\ddot R_k|^2 + |\ddot c_k|^2)\right)^{1/2}.
\end{equation}
The last term can be bounded by $\sqrt{K_\infty}$ according to \ref{it_H1_bound_tensors}.
Integrating on the time interval $(0,t)$, $t<T$, it yields
\begin{equation*}
  \label{eq:bound_R_3}
  \int_0^t \left|\dfrac{\dot R_k}{R_k} \right| \mathrm{d}t \leq
  \dfrac{1}{\mu_g M_\infty d_\infty}
 T +
    \dfrac{{C_2}}{\mu_g}{ \left( 1 + \dfrac{1}{\sqrt{M_\infty d_\infty}} \right)}\sqrt{K T}.
\end{equation*}
Considering a smaller time $T$, only depending on $\mu_g$, $d_\infty$,
$M_\infty$, $C_0$ and $K_\infty$, it holds
\begin{equation*}
  \label{eq:bound_R_4}
  \int_0^t \left|\dfrac{\dot R_k}{R_k}\right| \mathrm{d}t <\dfrac 1 2,
\end{equation*}
which gives 
\[
\dfrac{R_k^0}{2}<e^{-1/2}R_k^0<R_k<e^{1/2}R_k^0<2R_k^0. 
\]
Finally the Assumption
\ref{it_bound_R0} on the initial radii leads to the desired estimate
$(\mathscr Q_1)$. We note in passing that we obtained the following lemma:
  \begin{Lem}
    \label{lem_RkpsRk}
    There exists a constant $\tilde{K}$ depending on $\mu_g$, $d_\infty$,
$M_\infty$, $C_0$ and $K$, such that:
\[
\int_0^T \left(\max_{k=1,\ldots,N} \left| \dfrac{\dot R_k(t)}{R_k(t)}\right|\right)^2 {\rm d}t \leq \tilde{K}.
\]
  \end{Lem}
The proof of this  lemma is a straightforward application of
\eqref{eq_bound_R_2} and is left to the reader.

\subsection*{Strict version $(\mathscr Q_2)$ of \ref{it_bound_F}}
First, we remark that we can also adapt the previous proof to yield the following lemma:
\begin{Lem}
  \label{lem_uf}
  There exists a constant $C'$, depending only on $K$ and the list of parameters \eqref{eq_listparametre}, such that, for $T <1,$ there holds
  \begin{equation*}
    \int_0^T \|\p_x u_f\|_{L^\infty(\mathcal F)} \mathrm{d}t\leq C'\sqrt{T}.
  \end{equation*}
\end{Lem}
\begin{proof}
  We use the $L^\infty$ bound on $\p_x u_f$, see
  \eqref{eq_Linf_dxuf}, and the bound \ref{it_bound_rho} on the
  density $\rho_f$. It holds, by integrating on $(0,T)$,
  \begin{equation*}
    \label{eq:bound_rho_2}
    \int_0^T \|\p_x u_f \|_{L^\infty(\mathcal F)} \mathrm{d}t\leq
    \dfrac{C}{\mu_f} \int_0^T \|\tilde \Sigma_f\|_{H^1(\Omega)} \mathrm{d}t+ T
    \max_{\underline{\rho}_\infty/2\leq r\leq2 \bar\rho_\infty} \mathrm{p}_f(r).
  \end{equation*}
  Inequality \ref{it_H1_bound_tensors} on the stress tensor leads to
  \begin{equation*}
    \label{eq:bound_rho_3}
    \int_0^T \|\p_x u_f \|_{L^\infty(\mathcal F)} \mathrm{d}t\leq
    \dfrac{C}{\mu_f} \sqrt{TK}+ T
    \max_{\underline{\rho}_\infty/2\leq r\leq 2\bar\rho_\infty} \mathrm{p}_f(r),
  \end{equation*}
  which gives the expected bound for $T<1$.
\end{proof}
   
The continuity of the velocities \eqref{eq_continuity_velocity}
implies then that
\begin{equation*}
  \label{eq:bound_Fk1}
  \begin{aligned}
    \dfrac{\mathrm{d}}{\mathrm{d}t}(x_{k+1}^- - x_k^+)&=
    u(x_{k+1}^-)-u(x_k^+)\\
    & \leq \|\p_x u_f\|_{L^\infty(\mathcal F)}|x_{k+1}^- - x_k^+|.
  \end{aligned}
\end{equation*}
From Lemma \ref{lem_uf}, we can choose $T$ small (depending only on $C'$) 
such that there holds:
\[
\dfrac{|\mathcal F_k^0|}{2}<|\mathcal F_k|<2 |\mathcal F_k^0|,
\]
on $(0,T),$ which leads to the desired estimate.

\subsection*{Strict version $(\mathscr Q_3)$ of \ref{it_bound_rho}}

Since the fluid density $\rho_f$ satisfies a continuity equation associated with the velocity $u_f$
on the fluid domains $\mathcal{F}_k$
which are transported by the same  velocity field $u_f$, a classical estimate on
$(0,T)$ provides
  \begin{multline}
    \label{eq_bound_rho_1}
   ( \min_{x\in \mathcal F^0}\rho_f^0) \times \exp{\left(-\int_0^T \|\p_x
          u_f\|_{L^\infty(\mathcal F)}\mathrm{d}t\right)}\\
      \leq \rho_f(t,x) \leq
      (\max_{x\in \mathcal F^0}\rho_f^0) \times \exp{\left(\int_0^T \|\p_x
          u_f\|_{L^\infty(\mathcal F)}\mathrm{d}t\right)},
    \end{multline}
  Lemma \ref{lem_uf} allows to bound the exponential terms in
  \eqref{eq_bound_rho_1} for small time.
Namely, for $T$ small (depending on $C'$)  it holds, on $[0,T]$,
  \begin{equation*}
    \label{eq:bound_tho_4}
    \rho_f(t,x) \in \bigg( \dfrac 1 2 \min_{x \in \mathcal F^0}\rho_f^0,
    2 \max_{x \in \mathcal F^0} \rho_f^0
   \bigg) .
  \end{equation*}
  Then the assumption \eqref{it_bound_rho0} on the initial fluid density
  allows to deduce a strict version of estimate \ref{it_bound_rho}.

\subsection*{Strict version $(\mathscr Q_4)$ of \ref{it_bound_H1}}
Applying \eqref{eq_bound_H1_1}, we obtain:
 \begin{align*}
   &  \sup\limits_{[0,T]} \left( \int_{\mathcal F} \right.\left.\mu_f \dfrac{|\p_x
         u_f|^2}{2} \mathrm{d}x+ \mu_g\sum_{k=1}^N \dfrac{|\dot
         R_k|^2}{R_k}\right)\\
     &+ \int_0^T \left(\int_{\mathcal F}\rho_f |\p_t u_f + u_f \p_x
       u_f|^2 \mathrm{d}x
     + \sum_{k=1}^N m_k (|\ddot c_k|^2 + |\ddot
     R_k|^2)\right)\\
   &\leq  \sup\limits_{[0,T]} \left[\left( 2\sum_{k=1}^N \kappa_k \dfrac{|\dot R_k|}{R_k}\right)
   +   \int_{\mathcal F} \mathrm{p}_f(\rho_f) |\p_x u_f| \mathrm{d}x\right] \\
   & + \int_0^T \int_{\mathcal F} \mu_f \dfrac{|\p_x u_f|^3}{2}\mathrm{d}x\\
   &+    \int_0^T \sum_{k=1}^N \left( 2\kappa_k \dfrac{|\dot R_k|^2}{R_k^2} + \mu_g \dfrac{|\dot
       R_k|^3}{ R_k^2}\right) + E_1,
 \end{align*}
with a constant $E_1$ depending only on the list of parameters \eqref{eq_listparametre}.
To proceed, we detail now the controls of the  five remaining terms on the right-hand side.

\medskip

Concerning the first line in the right-hand side,  the first term can be rewritten with \ref{it_bound_mk0}:
\begin{equation*}  \label{eq:bound_H1_2}
  \begin{aligned}
    \sum_{k=0}^N \kappa_k \dfrac{|\dot R_k|}{R_k}\leq \dfrac{1}{M_\infty} \sum_{k=0}^N \left(\dfrac{|\dot
        R_k|}{\sqrt{R_k}} \dfrac{1}{N \sqrt{R_k}} \right).
  \end{aligned}
\end{equation*}
The Cauchy--Schwarz inequality gives then
\begin{equation*}
  \label{eq:bound_H1_3}
  \begin{aligned}
    \sum_{k=0}^N \kappa_k \dfrac{|\dot R_k|}{R_k}\leq
    \dfrac{1}{M_\infty\sqrt{\mu_g}} \left(\mu_g \sum_{k=0}^N
      \dfrac{|\dot R_k|^2}{{R_k}}\right)^{1/2} \left( \sum_{k=0}^N
      \dfrac{1}{N^2{R_k}} \right)^{1/2}.
  \end{aligned}
\end{equation*}
The first parenthesis can be bounded by $K$ using \ref{it_bound_H1} and the
second one  by $1/\sqrt{d_\infty}$ thanks to \ref{it_bound_R}  so that
\begin{equation}
  \label{eq:bound_H1_4}
  \begin{aligned}
    \sum_{k=0}^N \kappa_k  \dfrac{|\dot R_k|}{R_k}\leq
    \dfrac{1}{M_\infty\sqrt{\mu_g}} \dfrac{\sqrt{K}}{\sqrt{d_\infty}}.
  \end{aligned}
\end{equation}
The control of the second pressure term 
relies on \ref{it_bound_rho}-\ref{it_bound_H1} and a Cauchy-Schwarz inequality :
\begin{align}
  \notag 
  \int_\mathcal{F} \mathrm{p}_f(\rho_f)|\p_x u_f| \mathrm{d}x & \leq
                                                                \dfrac{\sqrt{2}}{\sqrt{\mu_f}} \left( \int_{\mathcal F}\mu_f
                                                                \dfrac{|\p_x u_f|^2}{2}\mathrm{d}x\right)^{1/2} \left[\max \limits_{
                                                                [\underline{\rho}_\infty/2, 2\bar \rho_\infty]} \mathrm{p}_f\right]\sqrt{2}, \\
  \label{eq:bound_H1_4bis} & \leq
                             \dfrac{2}{\sqrt{\mu_f}}\left[\max \limits_{
                                                                [\underline{\rho}_\infty/2, 2\bar \rho_\infty]} \mathrm{p}_f\right]\sqrt{K}.
\end{align}
As for the term on the second line in the right-hand side of
\eqref{eq_bound_H1_1}, we decompose as follows
\begin{equation*} 
  \int_0^T \int_{\mathcal F}\mu_f \dfrac{|\p_x
    u_f|^3}{2}\mathrm{d}x \leq \left( \sup_{[0,T]}\int_{\mathcal F}\mu_f
    \dfrac{|\p_x u_f|^2}{2}\right)\left( \int_0^T \|\p_x
    u_f\|_{L^\infty(\mathcal F)}\right),
\end{equation*}
where the first term can be bounded by $K$ according to
\ref{it_bound_H1}.  The second one is bounded using Lemma \ref{lem_uf}.
It follows that
\begin{equation}
  \label{eq:bound_H1_5}
  \int_0^T \int_{\mathcal F}\mu_f \dfrac{|\p_x
    u_f|^3}{2}\mathrm{d}x \leq
  C' \sqrt{T}K.
\end{equation}
We now turn to the first term on the third line. Applying a standard
$L^{\infty}-L^{1}$ H\"older inequality allows to bound this term by
\begin{equation*}
  \int_0^T \sum_{k=1}^N \kappa_k \dfrac{|\dot R_k|^2}{R_k^2}\leq
  \dfrac{1}{M_\infty\mu_g N} \max_{k\in \{1,\dots,N\}}
  \left\|\dfrac{1}{R_k}\right\|_{L^\infty(0,T)}\int_0^T \mu_g \sum_{k=1}^N
  \dfrac{|\dot R_k|^2}{R_k}.
\end{equation*}
The $L^\infty$ norm can be handled by the bound \ref{it_bound_R} and
the integral term by  \eqref{eq_Q6}.
It follows that
\begin{equation}
   \label{eq:bound_H1_8}
  \int_0^T \sum_{k=1}^N \kappa_k \dfrac{|\dot R_k|^2}{R_k^2}\leq
 \dfrac{E_0}{M_\infty\mu_g d_\infty}.
\end{equation}
It remains to bound the second term on the third line. In this
respect, we decompose the nonlinear term
\[
  \dfrac{|\dot{R}_k|^3}{R_k^2} = \dfrac{|\dot{R}_k|}{R_k}
  \dfrac{|\dot{R}_k|^{3/2}}{R_k^{3/4}}
  \dfrac{|\dot{R}_k|^{1/2}}{R_k^{1/4}}
\]
and apply a $L^{\infty}-L^{4/3}-L^{4}$ H\"older inequality to yield:
\begin{equation*}
   \label{eq:bound_H1_9}
   \sum_{k=1}^N \dfrac{|\dot R_k|^3
   }{R_k^2}\leq
   \max_{k\in \{1,\dots,N\}} \dfrac{|\dot R_k|}{R_k}
   \left( \sum_{k=1}^N\dfrac{|\dot R_k|^2
     }{ R_k}\right)^{3/4}
    \left( \sum_{k=1}^N\dfrac{|\dot R_k|^2
}{ R_k}\right)^{1/4}.
\end{equation*}
Integrating over $(0,T)$ we obtain again with a $L^2-L^{4}-L^4$
H\"older inequality that:
\[
  \int_0^T\sum_{k=1}^N \dfrac{|\dot R_k|^3
  }{ R_k^2}
  \leq
  \left( \int_0^T\left|\max_{k\in \{1,\dots,N\}} \dfrac{|\dot R_k|}{R_k}\right|^2\right)^{1/2}
  \left( \int_0^T \left|\sum_{k=1}^N\dfrac{|\dot R_k|^2
      }{ R_k}\right|^3\right)^{1/4}
  \left( \int_0^T\sum_{k=1}^N\dfrac{|\dot R_k|^2
    }{ R_k}\right)^{1/4}.
\]
Corollary \ref{cor_est_sigmag} and inequality \ref{it_bound_R} allow
to control the first term on the right-hand side. Indeed,
\begin{equation*}
   \mu_g \max_{k\in \{1,\dots,N\}} \dfrac{|\dot
          R_k|}{R_k} 
    \leq
    \dfrac{1}{M_\infty d_\infty}
  + C_1\bigg( \|\tilde
          \Sigma_f\|^2_{H^1(\Omega)}+ \dfrac{1
            }{M_\infty d_\infty} \sum_{k=1}^N m_k(|\ddot R_k|^2+ |\ddot
          c_k|^2)\bigg)^{1/2}.
\end{equation*}
Taking the $L^2$-norm in time and applying a triangular inequality and
\ref{it_H1_bound_tensors} provides
\begin{equation*}
    \left( \int_0^T \left| \max_{k\in \{1,\dots,N\}} \dfrac{|\dot
          R_k|}{R_k}\right|^2\right)^{1/2} \\
    \leq
    \dfrac{1}{\mu_g}\left(\sqrt{T}\dfrac{1}{M_\infty d_\infty}
      +C_1 \sqrt{K\left( 1+ \dfrac{1}{M_\infty d_\infty}\right)}\right).
\end{equation*}
As the second term is concerned, it holds
\begin{equation*}
  \begin{aligned}
    \left( \int_0^T \left|\sum_{k=1}^N\dfrac{|\dot R_k|^2 }{
          R_k}\right|^3\right)^{1/4}\leq T^{1/4}\left(
      \sup_{[0,T]}\sum_{k=1}^N \dfrac{|\dot R_k|^2}{R_k}\right)^{3/4},
  \end{aligned}
\end{equation*}
which can be handled thanks to \ref{it_bound_H1}, leading to
\begin{equation*}
  \left( \int_0^T \left|\sum_{k=1}^N\dfrac{|\dot R_k|^2 }{
        R_k}\right|^3\right)^{1/4}\leq T^{1/4}\left(\dfrac{K}{\mu_g}\right)^{3/4}.
\end{equation*}
Now the bound \eqref{eq_Q6} gives
\begin{equation*}
   \left( \int_0^T\sum_{k=1}^N\dfrac{|\dot R_k|^2
    }{ R_k}\right)^{1/4} \leq \left( \dfrac{E_0}{\mu_g}\right)^{1/4}.
\end{equation*}
To sum up, it finally yields
\begin{equation} \label{eq:bound_H1_13}
    \int_0^T\sum_{k=1}^N \dfrac{|\dot R_k|^3
  }{ R_k^2}\leq \dfrac{1}{\mu_g^2}T^{1/4} K^{3/4}E_0^{1/4}\left(\sqrt{T}\dfrac{1}{M_\infty d_\infty} + C_1\sqrt{K\left( 1+
    \dfrac{1}{M_\infty d_\infty}\right)} \right).
\end{equation}
Plugging
\eqref{eq:bound_H1_4}-\eqref{eq:bound_H1_4bis}-\eqref{eq:bound_H1_5}-\eqref{eq:bound_H1_8}-\eqref{eq:bound_H1_13}
into \eqref{eq_bound_H1_1}, it yields
\begin{multline}
  \label{eq:bound_H1_14}
  \sup_{[0,T]}\left(\int_{\mathcal F} \dfrac{|\p_x u_f|^2}{2}
    \mathrm{d}x+ \sum_{k=1}^N \mu_g \dfrac{|\dot
      R_k^2|}{R_k}\right)\\
  \begin{aligned}
    & \leq 2 \left( \dfrac{
        \sqrt{K}}{M_\infty \sqrt{\mu_gd_\infty}} +
      \dfrac{1}{\sqrt{\mu_f}} \max_{[\underline{\rho}_\infty/2, 2\bar
        \rho _\infty]}\mathrm{p}_f(r)\sqrt{K}\right)\\
    &+ C'\sqrt{T}{K}
    + \dfrac{ E_0}{M_{\infty}\mu_gd_\infty}  \\
    & +\dfrac{2}{\mu_g^2}T^{1/4} K^{3/4}E_0^{1/4}\left(\sqrt{T}\dfrac{1}{M_\infty d_\infty} + C_1\sqrt{K\left( 1+
          \dfrac{1}{M_\infty d_\infty}\right)} \right)+E_1.
  \end{aligned}
\end{multline}
If one considers $K>1$ and $T<1$,
defining
\begin{equation*}
  \begin{aligned}
    C_1' &= 2\left( \dfrac{1
        }{M_{\infty}\sqrt{\mu_gd_\infty}} +
      \dfrac{1}{\sqrt{\mu_f}} \max_{[\underline{\rho}_\infty/2, 2\bar
        \rho _\infty]}\mathrm{p}_f(r)\right),\\
    C_2' &=\dfrac{2}{\mu_g^2}E_0^{1/4}\left(\dfrac{
    1}{M_{\infty} d_\infty} + C_1\sqrt{ 1+
    \dfrac{1}{M_\infty d_\infty}} \right),
  \end{aligned}
\end{equation*}
the previous inequality writes
\begin{equation} \label{eq_toto}
    \sup_{[0,T]}\left(\int_{\mathcal F} \dfrac{|\p_x u_f|^2}{2}
      \mathrm{d}x+ \sum_{k=1}^N \mu_g \dfrac{|\dot
        R_k^2|}{R_k}\right)\leq
    C_1'\sqrt{K}  +\dfrac{E_0
      }{M_\infty \mu_g d_\infty}  
     + C'\sqrt{T}{K} + T^{1/4}K^{5/2}C_2' + E_1.
\end{equation}
Introduce $\lambda \in (0,1/2)$ to be fixed later on and set:
\[
  K_{\infty} := \dfrac{4|C'_1|^2}{\lambda^2} + \dfrac{1}
  {M_{\infty}\mu_g d_\infty} \dfrac{E_0}{2\lambda} + \dfrac{E_1}{1- 2\lambda}.
\]
When $K > K_{\infty},$ the sum of  the two first terms on the
right-hand side of \eqref{eq_toto} are bounded by $ \lambda K$.
Now taking $T$ small enough, for instance 
\[
  T= \min \{ (\lambda/2|C'|)^2, (\lambda (K^{3/2}C'_2)^{-1}/2)^4\},
\]
the sum of the third and fourth term can be bounded  by $\lambda K$
as well.
Finally,
the right-hand side is bounded according to
\begin{equation*}
  \begin{aligned}
    \sup_{[0,T]}\left(\int_{\mathcal F} \dfrac{|\p_x u_f|^2}{2}
      \mathrm{d}x+ \sum_{k=1}^N \mu_g \dfrac{|\dot
        R_k^2|}{R_k}\right)&\leq
    E_1+ 2 \lambda   K < K 
  \end{aligned}
\end{equation*}
since $\lambda<1/2$ and $C'_1 > 0.$

\subsection*{Strict version $(\mathscr Q_5)$ of \ref{it_H1_bound_tensors}}

In order to prove this estimate, we can adjust with the parameter
$\lambda$. First, thanks to Proposition \ref{prop_estima_sigmaTf} and
to the bounds \ref{it_bound_R} and \ref{it_bound_mk0}, there exists
$C>0$, depending in particular on $M_\infty$ and $d_\infty$, such that
  \begin{multline}
    \label{eq:H1_bound_tensors_1}
    \int_0^T \|\tSigma_f\|_{H^1(\Omega)}^2 \mathrm{d}t \leq C_0
    \int_0^T\Bigg[ \|\Sigma_f\|^2_{H^1(\mathcal F)} \\ + \sum_{k=1}^N
    m_k \big( |\ddot{R}_{k}|^2+ |\ddot{c}_k|^2\big) + \sum_{k=1}^N
    \bigg( \mu_g^2 \dfrac{|\dot{R}_k|^2}{R_k} +
    \dfrac{\kappa_k^2}{R_k}\bigg) \Bigg] \mathrm{d}t.
  \end{multline}
  The second term of the right-hand side can be bounded with the help of 
  \eqref{eq_bound_H1_1} by bounding the right-hand side of \eqref{eq_bound_H1_1}
 as in the previous analysis on \ref{it_bound_H1}. This entails:
  \begin{equation}
    \label{eq_mRC3}
    \int_0^T \sum_{k=1}^N m_k \big( |\ddot{R}_{k}|^2+ |\ddot{c}_k|^2\big)
    \mathrm{d}t \leq E_1 + 2 \lambda K.
  \end{equation}
  The third term is controlled using \eqref{eq_Q6}. The last term can
  be bounded by $T/((M_\infty)^2d_\infty)$. Therefore, this
  inequality becomes
 \begin{equation}   \label{eq:H1_bound_tensors_3}
   \int_0^T \|\tSigma_f\|_{H^1(\Omega)}^2 \mathrm{d}t \leq C_0
   \Bigg[ \int_0^T \|\Sigma_f\|^2_{H^1(\mathcal F)}\mathrm{d}t
  +   (E_1 +  2\lambda K) + \mu_g E_0 + \frac{T}{(M_\infty)^2d_\infty} 
   \Bigg] .
 \end{equation}
 Let us now focus on the first term.
 We use the definition \eqref{eq_fluid_tensor_sigmaf}  of
$\Sigma_f$ and the momentum
equation~\eqref{eq_fluid_NS_1D} to write
\begin{equation*}
  \begin{aligned}
    \int_0^T \|\Sigma_f\|^2_{H^1(\mathcal F)}\mathrm{d}t &= \int_0^T
    \| \mu_f \p_x u_f -\mathrm{p}_f(\rho_f) \|^2_{L^2(\mathcal
      F)}\mathrm{d}t \\
    &\quad + \int_0^T \| \rho_f(\p_tu_f + u_f \p_x u_f)
    \|^2_{L^2(\mathcal F)}\mathrm{d}t .
  \end{aligned}
\end{equation*}
Thanks to \eqref{eq_Q6} and \ref{it_bound_rho}, the first term can be
bounded. For the second term, one has
\begin{equation*}
  \int_0^T \| \rho_f(\p_tu_f + u_f \p_x u_f)
  \|^2_{L^2(\mathcal F)}\mathrm{d}t  \leq
  \bar{\rho}_\infty \int_0^T \int_\mathcal{F} \rho_f |\p_tu_f +
  u_f \p_x u_f |^2
  \mathrm{d}t ,
\end{equation*}
and this right-hand side actually appears in \eqref{eq_bound_H1_1} and
thus, the previous estimate obtained to prove \ref{it_bound_H1} can be
used. This provides
\begin{equation*}
  \int_0^T \|\Sigma_f\|^2_{H^1(\mathcal F)}\mathrm{d}t \leq 2 \mu_f E_0
  + 2 T \left[ \max_{[\underline{\rho}_\infty,\bar
    \rho _\infty]}\mathrm{p}_f\right]^2  +   \bar{\rho}_\infty (E_1 + 2\lambda
  K) .
\end{equation*}
Gathering the previous estimates and after rearrangement,
\begin{multline*}
  \int_0^T \|\tSigma_f\|_{H^1(\Omega)}^2 \mathrm{d}t \leq C_0
  \Bigg[ \big( 2 \mu_f E_0 + (\bar{\rho}_\infty+1) E_1 + \mu_g E_0\big)
  \\
  + \bigg( 2  \left[ \max_{[\underline{\rho}_\infty,\bar
    \rho _\infty]}\mathrm{p}_f\right]^2 + \frac{1}{(M_\infty)^2d_\infty}
  \bigg) T +  2(\bar{\rho}_\infty+1)\lambda K
  \Bigg] 
\end{multline*}
holds. Now starting from Proposition~\ref{prop_tSigmag}, a similar
estimate can be proved for $\tSigma_g$. Then, using
again~\eqref{eq_mRC3}, one finally have
\begin{equation*}
    \int_0^T  \bigg[ \|\tSigma_f\|_{H^1(\Omega)}^2 +
    \|\tSigma_g\|_{H^1(\Omega)}^2 + m_k \big( |\ddot{R}_{k}|^2+
    |\ddot{c}_k|^2\big)  \bigg] \mathrm{d}t \leq C_1 + C_2 T + C_3
    \lambda K,
\end{equation*}
where $C_1$, $C_2$ and $C_3$ are positive and independent of $N$, $T$,
$\lambda$ and $K$. To conclude, it suffices to choose $\lambda$
sufficiently small  so that $\lambda \leq (4C_3)^{-1}$ and $K \geq
4C_1.$ We can then take $T$ smaller if necessary so that $C_2T \leq
K/4$.

\section{Analysis of the density equation} \label{sec_rhoH1}

This section is devoted to the proof of the following proposition:
\begin{Prop}
  \label{prop_rhoH1}
  Assume that $T >0$ and $(\rho_f,u_f,(c_k,R_k)_{k=1,\ldots,N})$ is a
  classical solution to
  \eqref{eq_fluid_mass}-\eqref{eq_droplet_tensor_sigmak}
  on $(0,T)$ -- complemented with initial conditions constructed as in
  \eqref{eq_initbubble}-\eqref{eq_initbubble2} -- that satisfies
  \ref{it_bound_R}--\ref{it_H1_bound_tensors}.
  Then, there exists strictly positive constants $K_1$ and $T_1$
  depending only on the list of parameters \eqref{eq_listparametre}
  and $K$ such that
  \[
    \|\rho_f(t)\|_{H^1(\mathcal F(t))} \leq K_1 \text{ on $(0,T_1).$}
  \]
\end{Prop}

Again, the main difficulty in obtaining this proposition is to make
the constant $K_1$ independent of the parameter $N.$
For this, we proceed as in Section \ref{sec_deriv-macr-model} and
interpret $\rho_f$ on $\mathcal F(t)$ as the trace of some
 global density defined on $\Omega.$ We notice here that, by
 assumption, we already have this property initially
 since we set 
 \[
   \rho_f(0,\cdot) = \rho_f^{0} \text{ on $\mathcal F^0$}
 \]
 with $\rho_f^0 \in H^1(\Omega).$
 To extend this property, we construct again extensions $\tilde{u}_f$  of fluid velocity-field ${u}_f$ and 
 $\tilde{\Sigma}_f$ of stress tensor $\Sigma_f$  with the same formula
 as in \eqref{eq_tilde_udfN} and \eqref{eq_tSigmaf} respectively.
 We then construct $\tilde{\rho}_f$
 \begin{equation}
   \label{eq_trhofN_cauchy3}
   \begin{cases}
     \p_t \tilde \rho_f + \tilde u_f \p_x \tilde \rho_f =
     -\dfrac{\tilde\rho_f}{\mu_f}\left( \tilde \Sigma_f +
       \mathrm{p}_f(\tilde \rho_f)\right),& \text{on }(0,T)\times
     \Omega , \\
     \tilde \rho_f(0,.)=  \rho_f^{0},& \text{on }
     \Omega , 
   \end{cases}
 \end{equation}
 By \ref{it_H1_bound_tensors} with {\bf Proposition
   \ref{eq_Linf_dxuf}} for the fluid part and
 \ref{it_bound_R}-\ref{it_H1_bound_tensors} with Corollary
 \ref{cor_est_sigmag} for the bubble part, we obtain that
 $\tilde{u}_f \in L^2(0,T;W^{1,\infty}(\Omega))$ with
 $\tilde{\Sigma}_f \in L^2(0,T;H^1(\Omega)).$
 Consequently, we have a unique solution \eqref{eq_trhofN_cauchy3}
 which solves \eqref{eq_fluid_mass} on $\mathcal F.$
 By uniqueness of the solution to \eqref{eq_fluid_mass} in the
 regularity class of classical solutions (see \cite{HMS1}),
 we have thus $\tilde{\rho}_f = \rho_f$ on $\mathcal F(t)$ for $t \in
 (0,T).$ So, our proof reduces to computing bounds for
 $\tilde{\rho}_f.$
 
 First, we prove that there exists $T_0\leq T$ such that we can control 
 $\|\tilde \rho_f\|_{L^{\infty}(\Omega)}$ explicitly on $(0,T_0).$  By the method of
 characteristics and the explicit value of ${\rm p}_f:$
 \[
   \|\tilde \rho_f(t,.)\|_{L^\infty(\Omega)}\leq \|
   \rho_f^{0}\|_{L^\infty(\Omega)} 
    \exp\left( \dfrac{1}{\mu_f}\int_0^T \bigg( \|\tilde
     \Sigma_f \|_{L^\infty(\Omega)} + a_f  \|\tilde
     \rho_f(t,.)\|_{L^\infty(\Omega)}^{\gamma_f}\bigg)
     \mathrm{d}t \right).
\]
 The bound \ref{it_H1_bound_tensors} coupled with the embedding of
 $H^1(\Omega)$ in $L^\infty(\Omega)$ allows to control the stress
 tensor norm by $K$. If
 $\|\tilde \rho_f(t,.)\|_{L^\infty(\Omega)}\leq 2 \|
 \rho_f^{0}(t,.)\|_{L^\infty(\Omega)}$, it yields
 \begin{equation*}
   \label{eq_trhofN3}
   \|\tilde \rho_f(t,.)\|_{L^\infty(\Omega)} \leq \|
   \rho_f^{0}\|_{L^\infty(\Omega)}\exp \left(
     \dfrac{1}{\mu_f}\sqrt{T K} + 2 a_f T \|
     \rho_f^{0}(t,.)\|_{L^\infty(\Omega)}^{\gamma_f} \right).
 \end{equation*}
 By a standard continuation argument, we construct then a time-interval
 $(0,T_0)$ depending only on $K$, $a_f$, $\gamma_f$ and
 $\|\rho_f^{0}(t,.)\|_{L^\infty(\Omega)}$ so that:
 \begin{equation*}
   \label{eq_trhofN4}
   \|\tilde \rho_f(t,.)\|_{L^\infty(\Omega)} \leq 2\|
   \rho_f^{0}\|_{L^\infty(\Omega)}
 \end{equation*}
 for $t<T_0$.
 
 \medskip
 
 We focus now on $\partial_x \tilde{\rho}_f.$ For this, we apply a
 space derivative to \eqref{eq_trhofN_cauchy3}:
 \begin{equation*}
   \begin{cases}
     \partial_t (\dv_x\tilde{\rho}_f) + 
     \partial_x \big(\tu \dv_x\tilde{\rho}_f \big) =
     - \dfrac{\tilde{\rho}_f}{\mu_f}  \partial_x
     \tilde{\Sigma}_f  \\
     \qquad        \qquad        \qquad        \qquad        \qquad
     - \dfrac{1}{\mu_f} \big(
     \tilde{\Sigma}_f +
     \mathrm{p}_f(\tilde{\rho}_f) +
     \tilde{\rho}_f \mathrm{p}_f'(\tilde{\rho}_f) \big) 
     \dv_x\tilde{\rho}_f  \\
     (\dv_x\tilde{\rho}_f)(0,\cdot) = \partial_x  {\rho}_f^{0}.
   \end{cases}
 \end{equation*}
 For simplicity, we denote from now on
 $Y:=\dv_x\tilde{\rho}_f$. We multiply the previous equation
 by $2Y$, leading to 
 \begin{equation*}
   \dv_t (Y^2) + \dv_x (\tu Y^2) = - 2Y
   \dfrac{\tilde{\rho}_f}{\mu_f}
   \partial_x\tilde{\Sigma}_f  - Y^2 A
 \end{equation*}
 where $A$ denotes
 $\partial_x \tu + \frac{2}{\mu_f} \big(
 \tilde{\Sigma}_f + \kappa_f (\gamma_f+1)
 (\tilde{\rho}_f)^{\gamma_f} \big)$. Let first bound the
 right-hand side by a standard Cauchy-Schwarz/Minkowski inequality:
 \[
   \int_\Omega \Big(- 2Y \dfrac{\tilde{\rho}_f}{\mu_f}  \partial_x 
   \tilde{\Sigma}_f - Y^2 A \Big) \mathrm{d}x 
   \leq \frac{1}{\mu_f} \|\tilde{\rho}_f \dv_x
   \tilde{\Sigma}_f  \|^2_{L^2(\Omega)} + \Big(
   \frac{1}{\mu_f} + \|A\|_{L^\infty(\Omega)} \Big)\|Y\|^2_{L^2(\Omega)} .
 \]
 Going back to the PDE for $Y^2$, the $L^2$ norm of
 $\dv_x\tilde{\rho}_f$ can be bounded as
 \begin{multline*}
   \|\dv_x\tilde{\rho}_f\|^2_{L^2(\Omega)} \leq \bigg( \|
   \partial_x  {\rho}_f^{0} \|^2_{L^2(\Omega)} + \frac{1}{\mu_f}
   \int_0^{T}  \|\tilde{\rho}_f \dv_x
   \tilde{\Sigma}_f  \|^2_{L^2(\Omega)} \mathrm{d}t \bigg) \\
   \times \mathrm{exp} \bigg(\frac{T}{\mu_f} + \int_0^{T}
   \Big(\|\partial_x \tu \|_{L^\infty(\Omega)} + \frac{2}{\mu_f} \big(
   \|\tilde{\Sigma}_f \|_{L^\infty(\Omega)} + \kappa_f (\gamma_f+1)
   \|\tilde{\rho}_f \|_{L^\infty(\Omega)}^{\gamma_f} \big) \Big)
   \bigg).
 \end{multline*}
 All the terms can be controlled using \ref{it_bound_rho} and
 \ref{it_H1_bound_tensors}, except $\int_0^{T}\|\partial_x
 \tu\|_{L^\infty(\Omega)}\mathrm{d}t$. This latter term can be
 bounded using lemmas \ref{lem_uf} and \ref{lem_RkpsRk} (corresponding respectively
 to the contributions of $\|\partial_x \tu\|_{L^{\infty}(\mathcal F)}$
 and $\|\partial_x \tu\|_{L^{\infty}(\Omega \setminus \mathcal F)}$).
 Then, for a sufficiently
 small time $T_1\leq T_0$,
 \[
   \int_0^{T_1}\|\partial_x \tu \|_{L^\infty(\Omega)}
   \mathrm{d}t < \frac12,
 \]
 so that on $(0,T_1):$
 \begin{multline*}
   \|\dv_x\tilde{\rho}_f \|^2_{L^2(\Omega)} \leq \Big( \|
   \partial_x  {\rho}_f^{0} \|^2_{L^2(\Omega)} +
   \frac{2}{\mu_f}\| {\rho}_f^{0} \|^2_{L^\infty(\Omega)}
   K \Big) \\
   \times \mathrm{exp} \Big( \frac{T_1}{\mu_f} + \frac12 +
   \frac{2}{\mu_f} \sqrt{T_1K} + T_1\kappa_f(\gamma_f+1)2
   ^{\gamma_f} \| \bar{\rho}_f^{0} \|^{\gamma_f}_{L^\infty(\Omega)}
   \Big) .
 \end{multline*}
 This completes the proof.

%

\end{document}